\documentclass[11pt]{amsart}

\usepackage{url}
\usepackage{color}
\usepackage{graphicx}
\usepackage{amssymb}
\usepackage[colorlinks=true]{hyperref}
\usepackage{rotating}
\usepackage[lofdepth,lotdepth]{subfig}
\newcommand{\rotxc}[1]{\begin{sideways}#1\end{sideways}}
\newcommand{\invert}[1]{\rotxc{\rotxc{#1}}}

\newcommand{\A}{\mathcal{A}}
\newcommand{\C}{\mathbb{C}}

\newcommand{\RR}{\mathbb{R}}
\newcommand{\R}{\mathbb{R}}

\newcommand{\ZZ}{\mathbb{Z}}
\newcommand{\B}{\mathcal{B}}
\newcommand{\I}{\mathcal{I}}
\newcommand{\bigmid}{\;\middle\vert\;}
\def\M{\mathcal{M}}
\def\O{\mathcal{O}}
\def\Le{\hbox{\invert{$\Gamma$}}}

\font\co=lcircle10

\def\jr{\rotatedown{\smash{\raise2pt\hbox{\co \rlap{\rlap{\char'005} \char'007}}
               \raise6pt\hbox{\rlap{\vrule height6.5pt}}
                \raise2pt\hbox{\rlap{\hskip4pt \vrule
          height0.4pt depth0pt
                width7.7pt}}}}}
\def\textcross{\ \smash{\lower4pt\hbox{\rlap{\hskip4.15pt\vrule height14pt}}
                \raise2.8pt\hbox{\rlap{\hskip-3pt \vrule height.4pt depth0pt
                                width14.7pt}}}\hskip12.7pt}

\def\textelbow{\ \hskip.1pt\smash{\raise2.75pt%
                \hbox{\co \hskip 4.15pt\rlap{\rlap{\char'004} \char'006}
                \lower6.8pt\rlap{\vrule height3.5pt}
                \raise3.6pt\rlap{\vrule height3.5pt}}
                \raise2.8pt\hbox{%
                  \rlap{\hskip-7.15pt \vrule height.4pt depth0pt
width3.5pt}%
                  \rlap{\hskip4.05pt \vrule height.4pt depth0pt
width3.5pt}}}
                \hskip8.7pt}

\DeclareMathOperator{\convex}{convex}

\newtheorem{theorem}{Theorem}[section]
\newtheorem{lemma}[theorem]{Lemma}
\newtheorem{proposition}[theorem]{Proposition}
\newtheorem{corollary}[theorem]{Corollary}
\newtheorem{conjecture}[theorem]{Conjecture}
\theoremstyle{definition}
\newtheorem{definition}[theorem]{Definition}
\newtheorem{example*}[theorem]{Example}

\theoremstyle{remark}
\newtheorem{remark}[theorem]{Remark}

\DeclareRobustCommand{\qedify}[1]{%
  \ifmmode \quad\hbox{#1}
  \else
    \leavevmode\unskip\penalty9999 \hbox{}\nobreak\hfill
    \quad\hbox{#1}%
  \fi
}
\newenvironment{example}{\begin{example*}\pushQED{\qedify{$\diamondsuit$}}}{\popQED\end{example*}}

\begin{document}

\begin{abstract}
We investigate the role that non-crossing partitions play in the study of positroids, a class of matroids introduced by Postnikov. 
We prove that every positroid can be constructed uniquely by choosing a non-crossing partition on the ground set, and then freely placing the structure of a connected positroid on each of the blocks of the partition. 
This structural result yields several combinatorial facts about positroids. We show that the face poset of a positroid polytope embeds in a poset of weighted non-crossing partitions. We enumerate connected positroids, and show how they arise naturally in free probability. Finally, we prove that the probability that a positroid on $[n]$ is connected equals $1/e^2$ asymptotically. 
\end{abstract}

\title{Positroids and non-crossing partitions}

\date{\today}
\thanks{
The first author was partially supported by the National Science Foundation CAREER Award DMS-0956178 and the SFSU-Colombia Combinatorics Initiative. 
The second author was supported by the EPSRC grant EP/I008071/1.  
The third author was partially supported by the National Science Foundation CAREER award
DMS-1049513.
}
\author{Federico Ardila}
\address{Mathematics Department, San Francisco State University, United States.}
\email{federico@sfsu.edu}

\author{Felipe Rinc\'on}
\address{Mathematics Institute, University of Warwick, United Kingdom.}
\email{e.f.rincon@warwick.ac.uk}

\author{Lauren Williams}
\address{Mathematics Department, University of California, Berkeley, United States.}
\email{williams@math.berkeley.edu}

\maketitle

\setcounter{tocdepth}{1}
\tableofcontents

\section{Introduction}
\label{sec:intro}

A \emph{positroid} is a matroid on an ordered set which can be represented
by the columns of a full rank $d \times n$ real matrix such that all its maximal 
minors are nonnegative.  Such matroids were first considered by 
Postnikov \cite{postnikov} in his study of the \emph{totally 
nonnegative part of the Grassmannian}.  In particular, 
Postnikov showed that positroids are in bijection with several
interesting classes of combinatorial objects, including
Grassmann necklaces, decorated permutations, $\Le$-diagrams, 
and equivalence classes of plabic graphs.  

Positroids have many nice matroidal properties.  They are closed
under restriction, contraction, and duality, as well as a cyclic 
shift of the ground set.  Positroid polytopes also have 
nice properties.  A general matroid polytope for a matroid
on the ground set $[n]$ can be described by 
using $2^n$ inequalities; in contrast, as we describe in Section \ref{sec:polytopes},
a positroid polytope for a rank $d$ positroid on $[n]$ can be described using $dn+n$ inequalities.

The main structural result of this paper shows the connection between positroids and non-crossing partitions. In Theorem \ref{th:pos-sum} we prove 
that the connected components of a positroid form a non-crossing partition of its ground set. 
Conversely, each positroid on $[n]$ can be uniquely constructed by choosing a non-crossing partition $(S_1,\dots,S_t)$ of  $[n]$, and then putting the structure of a connected positroid on each block $S_i$. The first statement was also discovered in \cite{OPS}, where it is stated without proof, and in \cite{Ford}.
We also give an alternative description of this non-crossing partition in terms of Kreweras complementation.

Our structural result 
allows us to enumerate connected positroids, as described in Theorem \ref{th:countconn}. Along the way, we show in Corollary \ref{cor:connSIF} that the connected positroids on $[n]$ are in bijection with the
stabilized-interval-free 
permutations on $[n]$; that is, the permutations $\pi$ such that 
$\pi(I) \neq I$ for all intervals $I \subsetneq [n]$.
We then show in Theorem \ref{th:ratio} that the proportion of positroids on $[n]$ which are connected is equal to 
${1}/{e^2}$ asymptotically.  This result is somewhat surprising in 
light of the conjecture \cite{MNWW} that ``most matroids are connected"; more specifically, that as $n$ goes to infinity, the ratio of connected matroids on $[n]$ to matroids on $[n]$ tends to $1$.

Our enumerative results on  positroids  also allow
us to make a connection to free probability. Concretely, we show that  if $Y$ is the random variable
$1+Exp(1)$, then the $n$th moment 
$m_n(Y)$ equals the number of positroids on $[n]$, 
and the $n$th free cumulant $k_n(Y)$ equals the number
of connected positroids on $[n]$.

We also obtain some results on the matroid polytope of a positroid. In Proposition \ref{r:inequalities} we state and prove an inequality description for positroid polytopes, which we learned from Alex Postnikov \cite{postnikovpersonal} and will appear in \cite{LP}.
More strongly, we show in Theorem \ref{thm:embed} that the face poset of a positroid polytope naturally embeds in a poset of weighted non-crossing partitions.

The structure of this paper is as follows.
In Section \ref{sec:prelim} 
we review the notion of a matroid, as well as the operations
of restriction, contraction, and duality.
In Section \ref{sec:positroids} we show that positroids are 
closed under these operations as well as a cyclic shift of the 
ground set. We also show that if $\{S_1,\dots,S_t\}$ is a 
non-crossing partition of $[n]$, and $M_i$ is a positroid on $S_i$,
then the direct sum of the $M_i$s is a positroid.
In Section \ref{sec:objects} we review Postnikov's
notion of Grassmann necklaces,
decorated permutations, $\Le$-diagrams, and plabic graphs, all of which
are combinatorial objects parameterizing positroids. We review some of the bijections between them.
In Section \ref{sec:polytopes} we turn our attention to positroid
polytopes, and provide a simple inequality description of them due to Postnikov.  
We also show that each face of a positroid polytope is a positroid polytope.
In Section \ref{sec:plabic} we explain how to read off the bases
and basis exchanges of a positroid from a corresponding plabic graph.
In Section \ref{sec:nc} we prove our main structural result on 
positroids, that the connected components of a positroid comprise
a non-crossing partition of the ground set. We also prove a converse to this result. The proofs of these results use plabic graphs as well as positroid polytopes. 
In Section \ref{sec:kreweras} we give an alternative description of the non-crossing partition of a positroid, relating the Kreweras complement of the partition to the positroid polytope.
In Section \ref{sec:embed} we define the poset of \emph{weighted non-crossing partitions}, and show that the face poset of a positroid polytope is embedded in it.
In Section \ref{sec:enumeration} we give our enumerative results
for positroids, and in Section \ref{sec:free}
we make the connection to free probability.

\section{Matroids}
\label{sec:prelim}

A matroid is a combinatorial object which unifies several notions of independence. Among the many equivalent ways of defining a matroid we will adopt the point of view of bases, which is one of the most convenient for the study of positroids and matroid polytopes. We refer the reader to \cite{Oxley} for a more in-depth introduction to matroid theory.

\begin{definition}
A \emph{matroid} $M$ is a pair $(E, \B)$ consisting of a finite set $E$ and a nonempty collection of subsets $\B=\B(M)$ of $E$, called the \emph{bases} of $M$, which satisfy the \emph{basis exchange axiom}: 
\begin{center}
 If $B_1, B_2 \in \B$ and $b_1 \in B_1 - B_2$, then there exists $b_2 \in B_2 - B_1$ such that $B_1 -  \{b_1\} \cup \{b_2\} \in \B$.
\end{center}
\end{definition}
The set $E$ is called the \emph{ground set} of $M$; we also say that $M$ is a matroid on $E$. 
A subset $F \subset E$ is called \emph{independent} if it is contained in some basis. All the maximal independent sets contained in a given set $A \subset E$ have the same size, which is called the \emph{rank} $r_M(A)=r(A)$ of $A$. In particular, all the bases of $M$ have the same size, which is 
called the rank $r(M)$ of $M$.

\begin{example}
Let $A$ be a $d \times n$ matrix of rank $d$ with entries in a field $K$, and denote its columns by ${\bf a_1}, {\bf a_2}, \dotsc, {\bf a_n} \in K^d$. The subsets $B \subset [n]$ for which the columns $\{{\bf a_i} \mid i \in B \}$ form a linear basis for $K^d$ are the bases of a matroid $M(A)$ on the set $[n]$. Matroids arising in this way are called \emph{representable}, and motivate much of the theory of matroids.
\end{example}

There are several natural operations on matroids.

\begin{definition}\label{def:sum}
Let $M$ be a matroid on $E$ and $N$ a matroid on $F$. The \emph{direct sum}
of matroids $M$ and $N$ is the matroid $M \oplus N$
whose underlying set is
the disjoint union of $E$ and $F$, and whose bases are the disjoint 
unions of a basis of $M$ with a basis of $N$.
\end{definition}

\begin{definition}
Given a matroid $M=(E,\B)$, the \emph{orthogonal} or \emph{dual matroid} $M^*=(E,\B^*)$ is the 
matroid on $E$ defined by 
$\B^* = \{E - B \mid B\in \B\}$.
\end{definition}

\begin{definition}
Given a matroid $M=(E,\B)$, and a subset $S$ of $E$, the \emph{restriction}
of $M$ to $S$, written $M|S$, is the matroid on the ground set $S$ whose independent sets 
are all independent sets of $M$ which are contained in $S$. Equivalently,
the set of bases of $M|S$ is
\[
\B(M|S) = \{B \cap S \ \mid \ B \in \B, \text{ and } |B \cap S| \text{ is maximal among all }B \in \B \}. 
\]
\end{definition}

The dual operation of restriction is contraction.

\begin{definition}
Given a matroid $M=(E,\B)$ and a subset $T$ of $E$, the \emph{contraction}
of $M$ by $T$, written $M/T$, is the matroid on the ground set $E - T$ whose bases
are the following:
\[
\B(M/T) = \{B - T \ \mid \ B \in \B, \text{ and } |B \cap T| \text{ is maximal 
among all }B \in \B \}. 
\]
\end{definition}

\begin{proposition}\cite[Chapter 3.1, Exercise 1]{Oxley}\label{prop:drc}
If $M$ is a matroid on $E$ and $S \subset E$, then 
\[
(M/S)^* = M^*|(E - S).
\]
\end{proposition}

\section{Positroids}
\label{sec:positroids}

In this paper we study a special class of representable matroids introduced by Postnikov in \cite{postnikov}. We begin by collecting several foundational results on positroids, most of which are known \cite{oh, postnikov}.

\begin{definition}\label{def:positroid}
Suppose $A$ is a $d \times n$ matrix of rank $d$ with real entries such that all its maximal minors are nonnegative. Such a matrix $A$ is called \emph{totally nonnegative}, 
and the representable matroid $M(A)$ associated to $A$ is called a \emph{positroid}.
\end{definition}

\begin{remark}
We will often identify the ground set of a positroid with the set $[n]$, but more 
generally, the ground set of a positroid may be any finite set $E=\{e_1,\dots,e_n\}$,
endowed with a specified total order 
$e_1<\dots<e_n$. Note that the fact that a given matroid is a positroid is strongly dependent
on the total order of its ground set; in particular, being a positroid is not invariant
under matroid isomorphism.
\end{remark}

If $A$ is as in Definition \ref{def:positroid} and $I \in \binom{[n]}{d}$ is a $d$-element subset
of $[n]$, then we let $\Delta_I(A)$ denote the $d \times d$ minor of $A$ 
indexed by the column set $I$.  These minors are called the
\emph{Pl\"ucker coordinates} of $A$.

In our study of positroids, we will repeatedly make use of the following notation. Given $k,\ell \in [n]$, we define the \emph{(cyclic) interval} $[k,\ell]$ to be the set
\[
[k,\ell] := 
\begin{cases}
 \{ k, k+1, \dotsc, \ell \} &\text{ if $k \leq \ell$},\\
 \{ k, k+1, \dotsc, n, 1, \dotsc, \ell \} &\text{ if $\ell < k$}.
\end{cases}
\]
We also refer to a cyclic interval as a \emph{cyclically consecutive} subset of $[n]$.
We will often put a total order on a cyclic interval: in the first case above
we use the total order $k < k+1 < \dots < \ell$, and in the second case,
we use the total order $k < k+1 < \dots < n < 1 < \dots < \ell$.

Positroids are closed under several key operations:

\begin{lemma}\label{lem:cyclic}
Let $M$ be a positroid on the ground set 
$E = \{1 < \dots < n\}$.  Then for any $1 \leq a \leq n$, 
$M$ is also a positroid
on the  ordered ground set $\{a < {a+1} < \dots < n < 1 < \dots < {a-1}\}$.
\end{lemma}
\begin{proof}
Let $M = M(A)$ for some totally nonnegative full rank 
$d \times n$ matrix $A$.  Write $A=(v_1,\dots,v_n)$
as a concatenation of its column vectors $v_i \in \R^d$.
Then, as noted in \cite[Remark 3.3]{postnikov},
the matrix 
$A' = (v_2, \dots, v_n, (-1)^{d-1} v_1)$ obtained by
cyclically shifting the columns of $A$ and multiplying the
last column by $(-1)^{d-1}$ is also totally nonnegative.
Moreover, 
$\Delta_I(A) = \Delta_{I'}(A')$, where $I'$ is the cyclic
shift of the subset $I$.  Therefore $M(A')$ is a positroid, which 
coincides with $M$ after cyclically shifting the ground set.
It follows that $M$ is a positroid on 
$\{2 < 3 < \dots < n < 1\}$, and by iterating this construction,
the lemma follows.
\end{proof}

\begin{proposition}\label{directsum}
Suppose we have a decomposition of $[n]$ into two 
cyclic intervals
$[\ell+1,m]$ and $[m+1,\ell]$.
Let $M$ be a positroid on the ordered ground set 
$[\ell+1, m]$
and let $M'$ be a positroid on the 
ordered ground set 
$[m+1,\ell]$.
Then $M \oplus M'$ is a positroid on the ordered ground 
set $[n] = \{1 < \dots < n\}$.
\end{proposition}

\begin{proof}
First assume $\ell=n$. Let $M$ be a positroid on the ground
set $[m]=\{1<\dots <m\}$ and  $M'$
be a positroid on the ground set 
$\{m+1 <\dots <n\}$.
Then $M=M(A)$ and $M'=M(B)$, where 
$A$ and $B$ are full rank 
$d \times m$ and $d' \times (n-m)$ matrices whose 
maximal minors are nonnegative.
Use $A$ and $B$ to 
form the $(d+d') \times n$ block matrix 
of the form 
\[
\begin{pmatrix}
A & 0 \\
0 & B
\end{pmatrix}.
\]
Clearly this matrix has all maximal minors nonnegative and represents
the direct sum $M(A) \oplus M(B)$ of the matroids $M(A)$ and $M(B)$.
It follows that $M \oplus M'$ is a positroid.
Now the proposition follows from Lemma \ref{lem:cyclic}.
\end{proof}

The following proposition says that 
positroids are closed under duality, restriction, and contraction.

\begin{proposition}\label{prop:closed}
Let $M$ be a positroid on $[n]$.  Then $M^*$ is also a positroid on $[n]$.
Furthermore, for any subset $S$ of $[n]$, the restriction
$M|S$ is a positroid on $S$, 
and the contraction $M/S$ is a positroid on 
$[n]-S$.  Here the total orders 
on $S$ and $[n]-S$ are the ones
inherited from $[n]$.  
\end{proposition}

\begin{proof}
Consider a full rank $d \times n$ real matrix $A$ such that 
$M = M(A)$ and all maximal minors of $A$ are nonnegative.
By performing row operations on $A$ and multiplying rows by $-1$ when necessary,
we may assume without loss of generality that $A$ is in reduced 
row-echelon form.
In particular, $A$ contains the identity matrix in 
columns $i_1, i_2,\dots, i_d$ for some $i_1< \dots< i_d$.
Let $J = \{i_1,\dots,i_d\}$ and $J^c = [n]-\{i_1,\dots,i_d\}$.
Let us label the rows of $A$ by $i_1,\dots, i_d$ from top to bottom, and the 
columns of $A$ by $1,2, \dots, n$ from left to right.
If the entry of $A$ in row $s$ and column $t$ is not determined
to be $0$ or $1$ by the row-echelon form
(here we have necessarily that $s<t$),
let us denote it by 
\begin{equation*}
(-1)^{q_{st}} a_{st}, \text{ where }
q_{st} = |\{s+1, s+2,\dots,t-1\} \cap J|.
\end{equation*}
See the first matrix in Example \ref{ex:duality}.

Now we construct an $(n-d) \times n$ matrix $A' = (a'_{ij})$, with 
rows labeled by $J^c=[n]-\{i_1,\dots,i_d\}$ from top to bottom, 
and columns labeled by $1,2,\dots,n$ from left to right, as follows.
First we place the identity matrix in columns $J$.
Next, we set to $0$ every entry 
of $A'$ which is in the same row as and to the right of a $1$.
For the remaining entries we define
$a'_{ij} = \pm a_{ji}$.  More specifically, 
for the entry  $a'_{ts}$ in row $t$ and column $s$
(here we have necessarily that $s<t$)  we set
\begin{equation*}
a'_{ts} = (-1)^{q'_{st}} a_{st}, \text{ where }
q'_{st} = |\{s+1, s+2,\dots,t-1\} \cap J^c|.
\end{equation*}
See the second matrix in Example \ref{ex:duality}.  It is not hard to check that for each 
$I \in \binom{[n]}{d}$, we have that 
 $\Delta_I(A) = \Delta_{[n]-I}(A')$.  It follows that 
$M(A')$ is the dual $M^*$ of $M$ and is also a positroid, as we wanted.

We will now prove that the contraction $M/S$ is a positroid on $[n] - S$.
If $S_1 \cap S_2 = \emptyset$ then $(M/S_1)/S_2 = M/(S_1 \cup S_2)$, so by induction it is enough to prove 
that $M/S$ is a positroid for $S$ a subset of size $1$.
Moreover, in view of Lemma \ref{lem:cyclic}, we can assume without
loss of generality that $S=\{1\}$.
Again, suppose that $A=(a_{ij})$ is a full rank $d \times n$ real matrix 
in reduced row-echelon form such that 
$M = M(A)$ and all maximal minors of $A$ are nonnegative.
If $\{1\}$ is a dependent subset in $M$ then the first column of $A$ contains only zeros,
and $M/S$ is the rank $d$ positroid on $S$ represented by the submatrix of $A$ obtained by eliminating its first column.
If $\{1\}$ is an independent subset in $M$ then the first column of $A$ is the vector $e_1 \in \R^d$.
The matroid $M/S$ is then represented by the submatrix $A'$ of $A$ obtained by eliminating its first column and its first row,
which also has nonnegative maximal minors since $\Delta_I(A') = \Delta_{\{1\} \cup I}(A)$.

Finally, since positroids are closed under duality and contraction, by 
Proposition \ref{prop:drc} they are also closed under restriction.
\end{proof}

\begin{example}\label{ex:duality}
Let \[A=
\begin{pmatrix}
0 & 1 & a_{23} & 0 & -a_{25} & -a_{26}\\
0 & 0 & 0 & 1 & a_{45} & a_{46}
\end{pmatrix}\]
represent a matroid $M(A)$ on $[6]$.
Then the matrix $A'$ as defined in the proof of Proposition
\ref{prop:closed} is given by 
\[A' = 
\begin{pmatrix}
1 & 0 & 0 & 0 & 0 & 0\\
0 & a_{23} & 1 & 0 & 0 & 0\\
0 & -a_{25} & 0 & a_{45} & 1 & 0\\
0 & a_{26} & 0 & -a_{46} & 0 & 1
\end{pmatrix},
\]
and $M(A') = M(A)^*$.  Moreover, 
for each $I \in \binom{[6]}{2}$, 
$\Delta_I(A) = \Delta_{[6]- I}(A')$.
\end{example}

\section{Combinatorial objects parameterizing positroids}
\label{sec:objects}

In \cite{postnikov}, Postnikov gave several families of combinatorial objects in bijection with positroids.   In this section we will start by defining
his notion of \emph{Grassmann necklace}, and explain how each one 
naturally labels a positroid.  We will then define
\emph{decorated permutations}, 
\emph{$\Le$-diagrams}, 
and equivalence classes of \emph{reduced plabic graphs}, and give
(compatible) bijections among all these objects.  This will give us a 
canonical way to label each positroid by a Grassmann necklace, 
a decorated permutation, a $\Le$-diagram, and a plabic graph.

\subsection{Grassmann necklaces}
\label{subsec:necklaces}

\begin{definition}
Let $d \leq n$ be positive integers. A \emph{Grassmann necklace} of type $(d,n)$ is a sequence $(I_1, I_2, \dotsc, I_n)$ of $d$-subsets $I_k \in \binom{[n]}{d}$ such that for any $i \in [n]$
\begin{itemize}
 \item if $i \in I_i$ then $I_{i+1} = I_i - \{i\} \cup \{j\}$ for some $j \in [n]$,
 \item if $i \notin I_i$ then $I_{i+1} = I_i$,
\end{itemize}
where $I_{n+1} = I_1$.
\end{definition}

The \emph{$i$-order} $<_i$ on the set $[n]$ is the total order 
\[
i \,<_i\, i+1 \,<_i\, \dotsb \,<_i\, n \,<_i\, 1 \,<_i\, \dotsb \,<_i\, i-2 \,<_i\, i-1.
\]
For any rank $d$ matroid $M = ([n], \B)$, let $I_k$ be the lexicographically minimal basis of $M$ with respect to the order $<_k$, and denote
\[
\I(M) := (I_1, I_2, \dotsc, I_n).
\]

\begin{proposition}[{\cite[Lemma 16.3]{postnikov}}]
For any matroid $M = ([n], \B)$ of rank $d$, the sequence $\I(M)$ is a Grassmann necklace of type $(d,n)$.
\end{proposition}

In the case where the matroid $M$ is a positroid we can actually recover $M$ from its Grassmann necklace, as described below.

Let $i \in [n]$. The \emph{Gale order} on $\binom{[n]}{d}$ (with respect to $<_i$) is the partial order $\leq_i$ defined as follows: for any two $d$-subsets $S = \{ s_1 <_i \dotsb <_i s_d\} \subset [n]$ and $T = \{ t_1 <_i \dotsb <_i t_d\} \subset [n]$, we have $S \leq_i T$ if and only if $s_j \leq_i t_j$ for all $j \in [d]$. 

\begin{theorem}[\cite{postnikov, oh}]\label{r:corres}
Let $\I = (I_1, I_2, \dotsc, I_n)$ be a Grassmann necklace of type $(d,n)$. Then the collection
\[
 \B(\I) := \left\{ B \in \binom{[n]}{d} \bigmid B \geq_j I_j \text{ for all } j \in [n] \right\}
\]
is the collection of bases of a rank $d$ positroid $\M(\I) := ([n], \B(\I))$. Moreover, for any positroid $M$ we have $\M(\I(M)) = M$.
\end{theorem}

Theorem \ref{r:corres} shows that $\M$ and $\I$ are inverse bijections between the set of Grassmann necklaces of type $(d,n)$ and the set of rank $d$ positroids on the set $[n]$. 

We record the following fact for later use. It follows directly from the definitions.

\begin{proposition}\label{cor:contain}
Let $M$ be a matroid. Then every basis of $M$ is also a basis of the positroid $\M(\I(M))$.
\end{proposition}

Note that for any matroid $M$, the positroid $\M(\I(M))$ is the smallest positroid containing $M$, in the sense that any positroid containing all bases of $M$ must also contain all bases of $\M(\I(M))$.

\subsection{Decorated permutations}

The information contained in a Grassmann necklace can be encoded in a more compact way, as follows. 

\begin{definition}
A \emph{decorated permutation} of the set $[n]$ is a bijection $\pi : [n] \to [n]$ whose fixed points are colored either ``clockwise'' 
or ``counterclockwise.'' 
We denote a clockwise fixed point by $\pi(j) = \underline{j}$ and 
a counterclockwise fixed point by $\pi(j) = \overline{j}$.
A \emph{weak $i$-excedance} of the decorated permutation $\pi$ is an element $j \in [n]$ such that either $j <_i \pi(j)$ or 
$\pi(j)=\overline{j}$ is a ``counterclockwise'' fixed point. 
The number of weak $i$-excedances of $\pi$ is the same for any $i \in [n]$; we will simply call it the number of \emph{weak excedances} of $\pi$.
\end{definition}

Given a Grassmann necklace $\I = (I_1, I_2, \dotsc, I_n)$ we can construct a decorated permutation $\pi_{\I}$ of the set $[n]$ in the following way.
\begin{itemize}
 \item If $I_{i+1} = I_i - \{i\} \cup \{j\}$ for $i \neq j$ then $\pi_{\I}(j):=i$.
 \item If $I_{i+1} = I_i$ and $i \notin I_i$ then $\pi_{\I}(i):=\underline{i}$.
 \item If $I_{i+1} = I_i$ and $i \in I_i$ then $\pi_{\I}(i):=\overline{i}$.
\end{itemize}
Conversely, given a decorated permutation $\pi$ of $[n]$ we can construct a Grassmann necklace $\I_{\pi} = (I_1, I_2, \dotsc, I_n)$ by letting $I_k$ be the set of weak $k$-excedances of $\pi$.
It is straightforward to verify the following.
\begin{proposition}
The maps $\I \mapsto \pi_{\I}$ and $\pi \mapsto \I_{\pi}$ are inverse bijections between the set of Grassmann necklaces of type $(d,n)$ and the set of decorated permutations of $[n]$ having exactly $d$ weak excedances.
\end{proposition}

\subsection{Le-diagrams}

\begin{definition}
Fix $d$ and $n$. For any partition $\lambda$, let
$Y_{\lambda}$ denote the Young diagram associated to $\lambda$.  
A {\it $\Le$-diagram}
(or Le-diagram) $D$ of shape $\lambda$ and type $(d,n)$
is a Young diagram $Y_{\lambda}$ contained in a $d \times (n-d)$ rectangle,
whose boxes are filled with $0$s and $+$s in such a way that the
{\it $\Le$-property} is satisfied:
there is no $0$ which has a $+$ above it in the same column and a $+$ to its
left in the same row.
See Figure \ref{fig:Le} for
an example of a $\Le$-diagram.
\end{definition}
\begin{figure}[ht]
\begin{center}
 \includegraphics[scale=0.9]{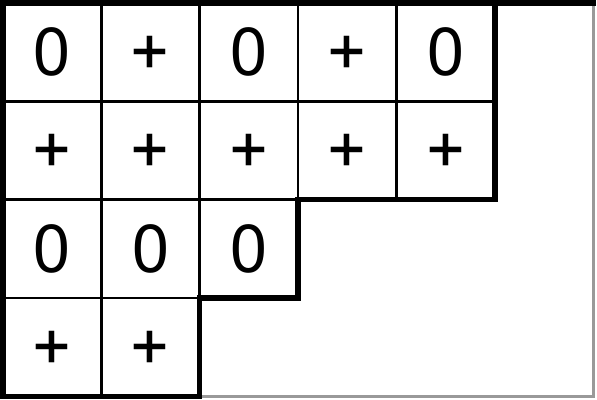}
 \caption{A Le-diagram with $\lambda=5532, d=4,$ and $n=10$.}
 \label{fig:Le}
\end{center}
\end{figure}

\begin{lemma}
The following algorithm is a bijection between $\Le$-diagrams of type 
$(d,n)$ and decorated permutations on $n$ letters with $d$ weak excedances.
\begin{enumerate}
\item Replace each $+$ in the $\Le$-diagram $D$
 with an elbow joint $\textelbow$, and each $0$ in $D$ with a cross
$\textcross$.
\item Note that the south and east
border of $Y_\lambda$ gives rise to a length-$n$
path from the northeast corner to the southwest corner of the
$d \times (n-d)$ rectangle.  Label the edges of this path with the
numbers $1$ through $n$.
\item Now label the edges of the north and west border of $Y_{\lambda}$
so that opposite horizontal edges and opposite vertical edges
have the same label.
\item View the resulting ``pipe dream"
as a permutation $\pi \in S_n$, by following the ``pipes" from 
the northwest border to the southeast border of the Young diagram.
If the pipe originating at label $i$ ends at the label $j$,
we define $\pi(i)=j$.
\item If $\pi(j)=j$ and $j$ labels two horizontal (respectively,
vertical) edges of $Y_{\lambda}$, then $\pi(j):=\underline{j}$
(respectively, $\pi(j):=\overline{j}$).
\end{enumerate}
\end{lemma}
Figure \ref{fig:pipedream} illustrates this procedure for the $\Le$-diagram of Figure \ref{fig:Le},
giving rise to the decorated permutation $\underline{1},7,9,3,2,\overline{6},5,10,4,8$.
\begin{figure}[ht]
\begin{center}
 \includegraphics[scale=0.9]{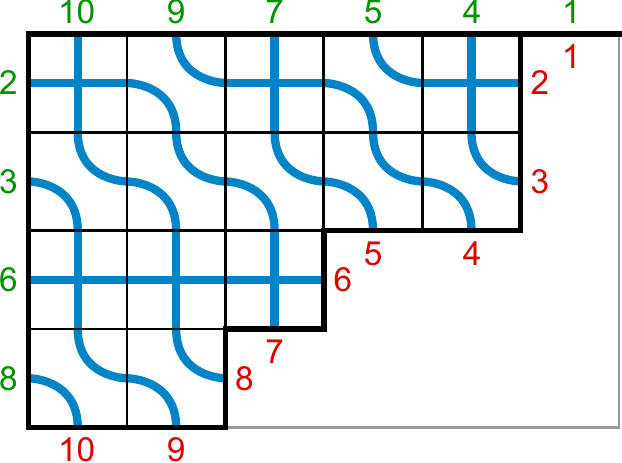}
 \caption{A ``pipe dream''.}
 \label{fig:pipedream}
\end{center}
\end{figure}

\subsection{Plabic graphs}
\begin{definition}
A {\it plabic graph\/}\footnote{``Plabic" stands for ``planar bi-colored."}  is an undirected graph $G$ drawn inside a disk
(considered modulo homotopy)
with $n$ {\it boundary vertices\/} on the boundary of the disk,
labeled $b_1,\dots,b_n$ in clockwise order, as well as some
colored {\it internal vertices\/}.
These  internal vertices
are strictly inside the disk and are
colored in black and white. Moreover, 
each boundary vertex $b_i$ in $G$ is incident to a single edge.

A {\it perfect orientation\/} $\O$ of a plabic graph $G$ is a
choice of orientation of each of its edges such that each
black internal vertex $u$ is incident to exactly one edge
directed away from $u$; and each white internal vertex $v$ is incident
to exactly one edge directed towards $v$.
A plabic graph is called {\it perfectly orientable\/} if it admits a perfect orientation.
Let $G_\O$ denote the directed graph associated with a perfect orientation $\O$ of $G$. 
The {\it source set\/} $I_\O \subset [n]$ of a perfect orientation $\O$ is the set of $i$ for which $b_i$
is a source of the directed graph $G_\O$. Similarly, if $j \in \overline{I}_{\O} := [n] - I_{\O}$, then $b_j$ is a sink of $\O$.
\end{definition}

Figure \ref{fig:orientation} shows a plabic graph with a perfect orientation. In that example, $I_{\O} = \{2,3,6,8\}$.

All perfect orientations of a fixed plabic graph
$G$ have source sets of the same size $d$, where
$d-(n-d) = \sum \mathrm{color}(v)\cdot(\deg(v)-2)$. 
Here the sum is over all internal vertices $v$, $\mathrm{color}(v) = 1$ for a black vertex $v$, 
and $\mathrm{color}(v) = -1$ for a white vertex;
see~\cite{postnikov}.  In this case we say that $G$ is of {\it type\/} $(d,n)$.

The following construction, which comes from \cite[Section 20]{postnikov},
associates a plabic graph to a $\Le$-diagram.

\begin{definition}\label{def:Le-plabic}
Let $D$ be a $\Le$-diagram. 
Delete the $0$s, and replace each $+$ with a vertex.
From each vertex we construct a hook which goes east
and south, to the border of the Young diagram.
The resulting diagram is called the ``hook diagram'' $H(D)$. After
replacing the edges along the south and east border of the Young diagram with  boundary vertices
labeled by $1,2,\dots,n$, we obtain a graph with $n$ boundary vertices and one internal vertex
for each $+$ from $D$.  Then
we replace the local region around each internal vertex
as in Figure \ref{fig:local}, and embed the resulting bi-colored graph
in a disk. Finally, for each clockwise (respectively, counterclockwise) fixed
point, we add a black (respectively, white) boundary ``lollipop'' at the corresponding
boundary vertex.  
This gives rise to a plabic graph which we call $G(D)$. 
\end{definition}

\begin{figure}[ht]
\begin{center}
 \includegraphics[scale=0.6]{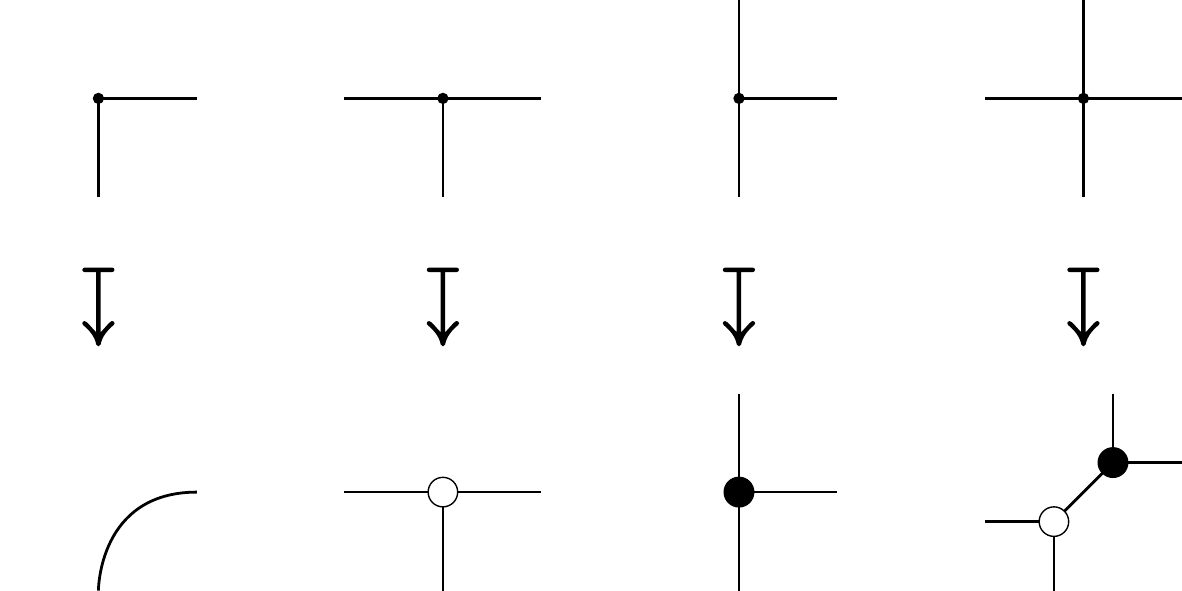}
 \caption{Local substitutions for getting the plabic graph $G(D)$ from the hook diagram $H(D)$.}
 \label{fig:local}
\end{center}
\end{figure}

Figure \ref{fig:hook} depicts the hook diagram corresponding to the 
$\Le$-diagram given in Figure \ref{fig:Le},
and Figure \ref{fig:plabic} shows its corresponding plabic graph.

\begin{figure}[ht]
  \centering
  \subfloat[][A hook diagram.]{
   \includegraphics[scale=0.9]{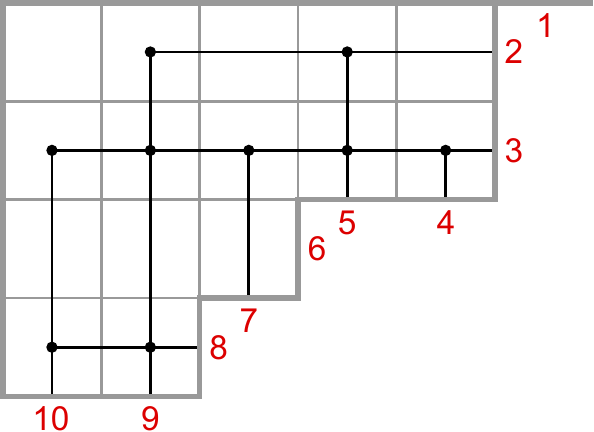}
   \label{fig:hook}
  }
  \qquad \quad
  \subfloat[][A plabic graph.]{
    \includegraphics[scale=0.9]{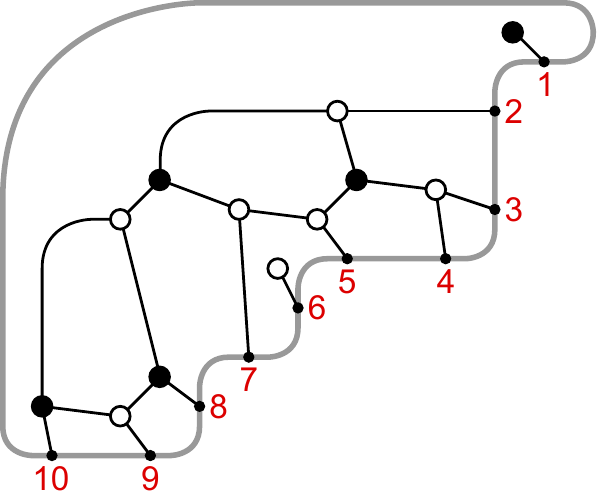}
   \label{fig:plabic}
  }
  \caption{}
\end{figure}

More generally each $\Le$-diagram $D$ 
is associated with a family of \emph{reduced plabic graphs}
consisting of $G(D)$ together with other plabic graphs which can be obtained
from $G(D)$ by certain \emph{moves}, see \cite[Section 12]{postnikov}.

From the plabic graph constructed in Definition \ref{def:Le-plabic} 
(and more generally from any 
leafless {reduced} plabic graph $G$ without isolated components), 
one may read off the corresponding decorated permutation $\pi_G$ as follows.

\begin{definition}\label{def:rules}
Let $G$ be a reduced plabic graph as above
with boundary vertices $b_1,\dots, b_n$.  
The \emph{trip} from $b_i$ is the path 
obtained by starting from $b_i$ and traveling along 
edges of $G$ according to the rule that each time we reach an internal white vertex we turn 
right, and each time we reach an internal black vertex we turn left. 
This trip ends at some boundary vertex $b_{\pi(i)}$.  
If the starting and ending points of the trip are the same vertex $b_j$, 
we set the color of the fixed point $\pi(j)=j$ to match the orientation of the trip
(clockwise or counterclockwise.)
In this way we associate a decorated permutation 
$\pi_G=(\pi(1),\dots,\pi(n))$ to each reduced plabic graph $G$, which 
is called the 
\emph{decorated trip permutation} of $G$.
\end{definition}

We invite the reader to verify that when we apply these rules to Figure \ref{fig:plabic} we obtain the trip permutation $\underline{1},7,9,3,2,\overline{6},5,10,4,8$.

\begin{remark}
All bijections that we have defined in this section are compatible.  This gives us
a canonical way to label each positroid of rank $d$ on $[n]$ by:
a Grassmann necklace, a decorated permutation, 
a $\Le$-diagram, and an equivalence
class of plabic graphs.
\end{remark}

\section{Positroid polytopes}\label{sec:polytopes}

The following geometric representation of a matroid will be
useful in our study of positroids.

\begin{definition} 
Given a matroid $M=([n],\B)$, the (basis) \emph{matroid polytope} $\Gamma_M$ of $M$ is the convex hull of the indicator vectors of the bases of~$M$:
\[
\Gamma_M := \convex\{e_B \mid B \in \B\} \subset \RR^n,
\]
where $e_B := \sum_{i \in B} e_i$, and $\{e_1, \dotsc, e_n\}$ is the standard basis of $\RR^n$.
\end{definition}

When we speak of ``a matroid polytope," we refer to the polytope of a specific matroid in its specific position in $\RR^n$. 

The following elegant characterization of matroid polytopes
is due to Gelfand,
Goresky, MacPherson, and Serganova.

\begin{theorem}[\cite{GGMS}]\label{r:GS} Let $\B$ be a collection of subsets of $[n]$ and let 
$\Gamma_{\B}:=\convex\{e_B \mid B \in \B\} \subset \RR^n$. Then $\B$ is the collection of bases of a matroid if and only if every edge of  $\Gamma_\B$ is a parallel translate of $e_i-e_j$ for some $i,j \in [n]$. 
\end{theorem}

When the conditions of Theorem \ref{r:GS} are satisfied, the edges of $\Gamma_{\B}$ correspond exactly to the basis exchanges; that is, to the pairs of distinct bases $B_1, B_2$ such that $B_2 = B_1 - \{i\} \cup \{j\}$ for some $i,j \in [n]$. Two such bases are called \emph{adjacent bases}.

The following result is a restatement of the greedy algorithm for matroids.

\begin{proposition}\label{prop:face}\cite[Prop. 2]{AK}
Let $M$ be a matroid on $[n]$.  
Then any face of the matroid polytope
$\Gamma_M$ is itself a matroid polytope. 
More specifically, let $w: \RR^n \to \R$ be a linear functional.
Let $w_i = w(e_i)$; note that by linearity, these values determine $w$.
Consider the flag of sets $\emptyset = A_0 \subsetneq A_1 \subsetneq \cdots \subsetneq A_k = [n]$ 
such that $w_a=w_b$ for $a,b \in A_i - A_{i-1}$, 
and $w_a < w_b$ for $a \in A_i-A_{i-1}$ and $b \in A_{i+1}-A_i$.
Then the face of $\Gamma_M$ minimizing the linear functional $w$ is the matroid polytope of the matroid 
\[
\bigoplus_{i=1}^{k} (M|A_i)/{A_{i-1}}.
\]
\end{proposition}

We now study inequality descriptions of positroid polytopes.

\begin{proposition}[\cite{welsh}]\label{r:inequalitiesmatroids}
Let $M = ([n], \B)$ be any matroid of rank $d$, and let $r_M:2^{[n]} \to \ZZ_{\geq 0}$ be its rank function. Then the matroid polytope $\Gamma_M$ can be described as
\[
\Gamma_M = \left\{ {\bf x} \in \RR^n \bigmid \sum_{i \in [n]} x_i = d, \, \sum_{i \in A} x_i \leq r_M(A) \, \text{ for all $A \subset [n]$} \right\}.
\]
\end{proposition}

Proposition \ref{r:inequalitiesmatroids} describes a general matroid polytope using the $2^n$ inequalities arising from the rank of all the subsets of its ground set. For positroid polytopes, however, there is a much shorter description, which we learned from Alex Postnikov \cite{postnikovpersonal}, and will appear in an upcoming preprint with Thomas Lam \cite{LP}. 

\begin{proposition}\label{r:inequalities}
Let $\I = (I_1, I_2, \dotsc, I_n)$ be a Grassmann necklace of type $(d,n)$, and let $M = \M(\I)$ be its corresponding positroid. For any $j \in [n]$, suppose the elements of $I_j$ are $a^j_1 <_j a^j_2 <_j \dotsb <_j a^j_d$. Then the matroid polytope $\Gamma_M$ can be described by the inequalities
\begin{align}
 x_1 + x_2 + \dotsb + x_n \, &= \, d, \label{rankeq} \\
 x_j \, &\geq \, 0 &\quad \text{ for all $j \in [n]$}, \label{0ineq}\\
 x_j + x_{j+1} + \dotsb + x_{a^j_k-1} \, &\leq \, k-1 &\quad \text{ for all $j \in [n]$ and $k \in [d]$} \label{Galeineq},
\end{align}
where all the subindices are taken modulo $n$.
\end{proposition}
In Proposition \ref{r:inequalities}, 
when we refer to taking some number $i$ modulo $n$, we mean taking its representative modulo $n$ in the set $\{1,\dotsc,n\}$.

\begin{proof}
Let $P$ be the polytope described by (\ref{rankeq}), (\ref{0ineq}), and (\ref{Galeineq}). First we claim that the vertices of $P$ are $0/1$ vectors. To see this, rewrite the polytope in terms of the ``$y$-coordinates" given by $y_i=x_1+\cdots+x_i$ for $1 \leq i \leq n-1$. The inequalities of $P$ are of the form $y_i-y_j \leq a_{ij}$ for integers $a_{ij}$. Since the matrix whose row vectors are $e_i-e_j$ is totally unimodular \cite{Schrijver}, the vertices of $P$ have integer $y$-coordinates, and hence also integer $x$-coordinates. The inequalities (\ref{0ineq}) and (\ref{Galeineq}) (for $k=2$) imply that the $x$-coordinates of any vertex are all equal to $0$ or $1$.

Since $P$ and $\Gamma_M$ are both $0/1$ polytopes, it suffices to show that they have the same vertices. But for a $0/1$ vector $e_B$ satisfying (\ref{rankeq}), the inequalities (\ref{Galeineq}) are equivalent to $B \geq_j I_j$ for all $j$, \emph{i.e.}, to $B \in \B(\I)$, as desired. 
\end{proof}

\begin{proposition}\label{prop:facets}
A matroid $M$ of rank $d$ on $[n]$ is a positroid if and only if its matroid polytope $\Gamma_M$ can be described by the equality $x_1+ \dotsb + x_n = d$ and inequalities of the form 
\[\sum_{\ell \in [i,j]} x_\ell \leq a_{ij}, \, \text{ with }i,j \in [n].\]
\end{proposition}
\begin{proof}
It follows from Proposition \ref{r:inequalities} that all positroid polytopes have the desired form (note that $x_j \geq 0$ is equivalent to $x_{j+1} + \dotsb + x_{j-1} \leq d$). To prove the converse, assume $M$ is a rank $d$ matroid on the set $[n]$ whose polytope $\Gamma_M$ admits a description as above.
Let $r_{ij} = r_M([i,j])$ be 
the rank in $M$ of the cyclic interval $[i,j]$. If $\Gamma_M$ satisfies an inequality $\sum_{\ell \in [i,j]} x_\ell \leq a_{ij}$ then $r_{ij} \leq a_{ij}$.
The polytope $\Gamma_M$ is then described by the inequalities
\begin{align*}
 x_1 + x_2 + \dotsb + x_n \, &= \, d,\\
 x_i + x_{i+1} + \dotsb + x_j \, &\leq \, r_{ij} & \text{for all $i,j \in [n]$}.
\end{align*}

Let $\I := \I(M) = (I_1, I_2, \dotsc, I_n)$ be the Grassmann necklace associated to $M$, and let $M' := \M(\I(M))$ be its corresponding positroid. We will show that $M=M'$. 

By Proposition \ref{cor:contain}, we know that every basis of $M$ is also a basis of $M'$. Now, suppose that $B$ is basis of $M'$, and consider any cyclic interval $[i,j]$. Denote $I_i =: \{ a_1 <_i a_2 <_i \dotsb <_i a_d\}$, and let $k = |I_i \cap [i,j]|$. Then $i \leq j \leq a_{k+1}-1$ in cyclic order. Combining this with Proposition \ref{r:inequalities}, we see that the vertex $e_B$ of $\Gamma_{M'}$ satisfies the inequality
\[
x_i + x_{i+1} + \dotsb + x_j \leq x_i + x_{i+1} + \dotsb + x_{a_{k+1}-1} \leq k.
\]
(with the convention that $a_{d+1} = i$.) Moreover, the definitions of $\I_i$ and $k$ imply that $k = r_{ij}$, showing that $e_B$ satisfies all the inequalities that describe $\Gamma_M$. It follows that $B$ is also a basis of $M$, as desired.
\end{proof}

It follows that positroid polytopes are closed under taking faces.

\begin{corollary}\label{cor:faceofpositroid}
Every face of a positroid polytope is a positroid polytope.
\end{corollary}
\begin{proof}
Assume $\Gamma_M$ is a positroid polytope, and fix a description of it by inequalities as in Proposition \ref{prop:facets}. Any face of $\Gamma_M$ is then obtained by intersecting $\Gamma_M$ with hyperplanes of the form 
$\sum_{\ell \in [i,j]} x_\ell = a_{ij}$. But this is equivalent to intersecting it with the halfspaces $\sum_{\ell \in [j+1,i-1]} x_\ell \leq d - a_{ij}$, so the result follows by Proposition \ref{prop:facets}.
\end{proof}

\section{Matroidal properties of positroids from plabic graphs}\label{sec:plabic}

As shown in \cite[Section 11]{postnikov}, 
every perfectly orientable plabic graph gives rise to a positroid as follows.
\begin{proposition}\label{prop:perf}
Let $G$ be a plabic graph of type $(d,n)$.
Then we have a positroid 
$M_G$ on $[n]$ whose bases are precisely
\[\{I_{\O} \mid \O \text{ is a perfect orientation of }G\},\]
where $I_\O$ is the set of sources of $\O$.
\end{proposition}

Moreover, every positroid can be realized in this way, using the 
construction described in 
Definition \ref{def:Le-plabic}.
(One perfect orientation of $G(D)$ may be obtained by 
orienting each horizontal edge in Figure \ref{fig:local} 
west, each vertical edge south, and each 
diagonal edge southwest.)

There is another way to read off from $G$ the bases of $M_G$, which 
follows from  Proposition \ref{prop:perf} and 
\cite[Theorem 1.1]{Talaska}.  
To state this result, we need to define the notion of a \emph{flow} in $G$.
For $J$ a set of boundary vertices with
$|J|=|I_{\O}|$, a {flow from $I_{\O}$ to $J$} is a collection of self-avoiding walks
and self-avoiding 
cycles, all pairwise vertex-disjoint, such that the sources of these walks are $I_{\O} - (I_{\O} \cap J)$
and the destinations are $J - (I_{\O} \cap J)$.

\begin{proposition}
Let $G$ be a plabic graph of type $(d,n)$.  
Choose a perfect orientation $\O$ of $G$.
Then the bases of the positroid $M_G$ are precisely
$$\{I \mid \text{there exists a flow from }I_{\O} \text{ to }I\}.$$
\end{proposition}

Not only can we read off bases from the plabic graph, we can also read 
off basis exchanges.  The backwards direction  of Proposition \ref{prop:basis} below
was observed in 
\cite[Section 5]{PSW}.  

\begin{proposition}\label{prop:basis}
Consider a positroid $M$ which is encoded by the perfectly orientable plabic graph $G$.
Consider a perfect orientation $\O$ of $G$ and let $I = I_{\O}$.  Then there is a basis
exchange between $I$ and $J = I - \{i\} \cup \{j\}$ if and only if
there is a directed path $P$ in $\O$ from the boundary vertex $i$ to the boundary 
vertex $j$.
\end{proposition}

\begin{proof}
Suppose that $P$ is a directed
path in $\O$ from  $i$ to $j$.
If we modify $\O$ by reversing all edges along $P$ we obtain another 
perfect orientation $\O'$, whose source set is $J = I - \{i\} \cup \{j\}$.
It follows from Proposition \ref{prop:perf} that $J$ is also a basis of $M$
and hence there is a basis exchange between $I$ and $J$ that swaps $i$ and $j$.

Conversely, let us suppose that there is a basis exchange between $I$ and $J = I - \{i\} \cup \{j\}$.
Then by Proposition \ref{prop:perf} there is a perfect orientation $\O'$
of $G$ such that $I_{\O'} = J$.  Comparing $\O'$ to $\O$, it is clear that the 
set of edges where the perfect orientations differ is a subgraph $H'$ of $G$ such that 
all vertices have degree $2$ (except possibly at the boundary), see e.g. 
\cite[Lemma 4.5]{PSW}.  
More specifically, $H$ is a disjoint union of some closed cycles $C_1, \dots, C_l$ 
together with a path $P$
between vertices $i$ and $j$, and $\O'$ is obtained from $\O$ by reversing all edges in $H$.
It follows from the definition of perfect orientation that $P$ must be a directed path
in both $\O$ and $\O'$.
\end{proof}

 \begin{figure}[ht]
  \centering
  \subfloat[][A perfect orientation.]{
   \includegraphics[scale=0.9]{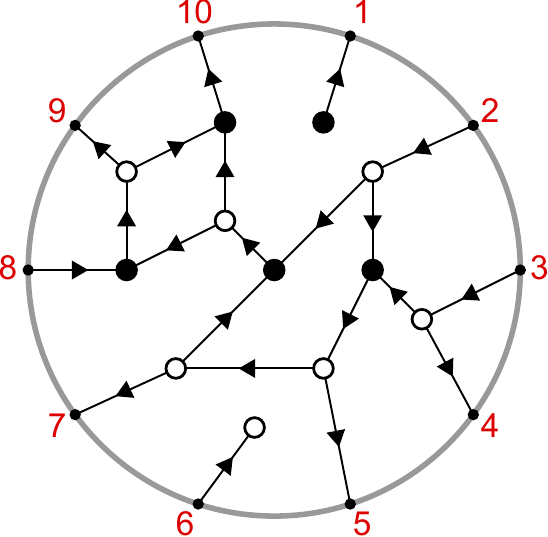}
   \label{fig:orientation}
  }
  \qquad\qquad
  \subfloat[][A flow from the set $\{2,3,6,8\}$ to the set $\{6,7,8,10\}$.]{
    \includegraphics[scale=0.9]{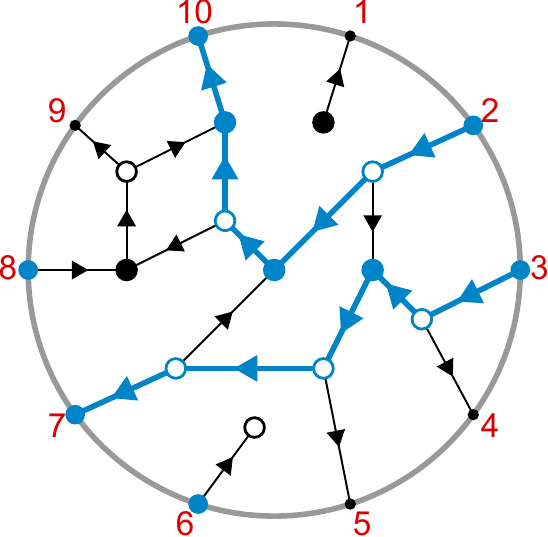}
   \label{fig:flow}
  }
  \\
  \subfloat[][A directed path from $3$ to $9$.]{
    \includegraphics[scale=0.9]{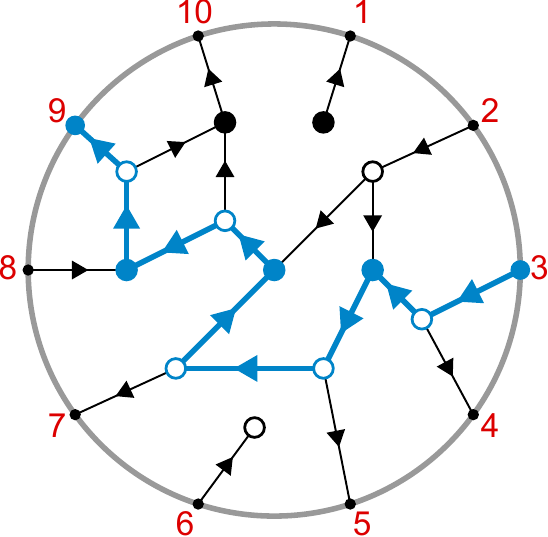}
   \label{fig:path}
  }
  \caption{A perfect orientation, a flow, and a directed path
in a plabic graph.}
\end{figure}

\begin{example}
 Figure \ref{fig:orientation} shows the plabic graph $G$ given in Figure \ref{fig:plabic} (rearranged without changing the combinatorial type), together with a perfect orientation $\O$ of its edges.
 The corresponding source set $I_{\O} = \{ 2,3,6,8 \}$ is then a basis of the corresponding positroid $M = \M(G)$.
 Figure \ref{fig:flow} depicts a flow from $I_{\O}$ to the set $I = \{6,7,8,10\}$, which implies that $I$ is a basis of $M$.
 The directed path in $\O$ from $3$ to $9$ highlighted in Figure \ref{fig:path} shows that the set $I_{\O} - \{3\} \cup \{9\}$ is a basis of $M$.
 Finally, since there is no directed path in $\O$ from $2$ to $4$, the set $I_{\O} - \{2\} \cup \{4\}$ is not a basis of $M$.
\end{example}

\section{Positroids, connected positroids, and  non-crossing partitions} 
\label{sec:nc}

In this section we begin to illustrate the role that non-crossing partitions play in the theory of positroids. More specifically, Theorem \ref{th:pos-sum}  shows that the connected components of a positroid form a non-crossing partition of $[n]$. 
Conversely, it also says that positroids on $[n]$ can be built out 
of connected positroids by first choosing a non-crossing
partition on $[n]$, and then putting the structure of a connected 
positroid on each of the blocks of the partition.

\begin{definition}
A matroid which cannot be written as the direct sum (see Definition \ref{def:sum}) of two nonempty
matroids is called \emph{connected}.
\end{definition}

\begin{proposition} \cite{Oxley}.
\label{prop:equiv}
Let $M$ be a matroid on $E$.  For two elements $a, b \in E$, we set 
$a \sim b$ whenever there are bases $B_1, B_2$ of $M$ such that 
$B_2 = B_1 - \{a\} \cup \{b\}$.  The relation $\sim$ is an equivalence
relation, and the equivalence classes are precisely the connected components of $M$.
\end{proposition}

\begin{proof}
It is more customary to define $a \sim b$ if $a, b \in C$ for some circuit $C$. It is known that this is an equivalence relation whose equivalence classes are the connected components of $M$ \cite[Chapter 4.1]{Oxley}.
We now verify that these two definitions are equivalent:

If $B_1$ and $B_2$ are bases of $M$ and $B_2 = B_1 - \{a\} \cup \{b\}$ then there is a unique circuit $C \subset B_1 \cup \{b\}$, called the \emph{fundamental circuit of $B$ with respect to $b$}. It has the property that $a,b \in C$. Conversely, if $a, b$ are contained in a circuit $C$, 
let $D$ be a basis of $M/C$. Then $B_2 := D \cup C - \{a\}$ and $B_1 := D \cup C - \{b\}$ are bases of $M$ such that $B_2 = B_1 - \{a\} \cup \{b\}$.
\end{proof}

\begin{lemma}
Let $M$ be a positroid on $E$, 
and write it as a direct sum of connected matroids
$M =  M_1 \oplus \dots \oplus M_l$.  Then each $M_i$ is a positroid.
\end{lemma}
\begin{proof}
This holds because each $M_i$ is the restriction of $M$
to some subset of $E$, and restrictions of positroids are positroids (Proposition \ref{prop:closed}.)
\end{proof}

\begin{proposition}\label{prop:pos-sum}
Suppose that $M$ is a positroid on $[n]$ which is the direct sum
$M_1 \oplus M_2$ of two connected positroids $M_1$ and $M_2$ on 
ground sets $E_1$ and $E_2$.  Then $E_1$ and $E_2$ are cyclic
intervals of $[n]$.
\end{proposition}

\begin{proof}
We propose two different arguments: one in terms of perfect orientations, and one in terms of matroid polytopes.

1. (Perfect orientations) Suppose that $E_1$ and $E_2$ are not cyclic intervals.  Then
there exist positive integers $1\leq i < j < k < l \leq n$ such that 
$i,k \in E_1$ and $j,l\in E_2$.
By Proposition \ref{prop:equiv}, there exist bases
$B_1$ and $B_1'$ of $M_1$ such that 
$i \in B_1$ and $B_1' = B_1 -\{i \} \cup \{k\}$,
and there exist bases
$B_2$ and $B_2'$ of $M_2$ such that 
$j \in B_2$ and $B_2' = B_2 -\{j \} \cup \{l\}$.
Then $B = B_1 \cup B_2$ is basis of $M$ which contains $i$ and $j$.
Moreover, 
$B$ admits a basis exchange which replaces $i$ with $k$,
and a basis exchange which replaces $j$ with $l$.

Therefore by Proposition \ref{prop:perf}
there exists a perfect orientation $\mathcal{O}$ of
a plabic graph for $M$ whose set $I_\mathcal{O}$ of sources contains $i$ and $j$.
And by Proposition \ref{prop:basis},  $\mathcal{O}$ has a directed path $P_1$ from $i$ to $k$,
and a directed path $P_2$ from $j$ to $l$.
Because $i<j<k<l$, these directed paths must intersect at some 
internal vertex $v$.  But now it is clear that $\mathcal{O}$
also contains a directed path $P_3$ from $i$ to $l$ (and from $j$ to $k$):
$P_3$ is obtained by following $P_1$ from $i$ to $v$, and then
following $P_2$ from $v$ to $l$.
Therefore $M$ has a basis exchange which switches $i$ and $l$,
contradicting our assumption that $i$ and $l$ lie in different
connected components of $M$.

2. (Matroid polytopes) The matroid polytope $\Gamma_M$ satisfies the equality 
\begin{equation}
\sum_{e \in E_1} x_e = r_M(E_1). \label{eq:E_1} 
\end{equation}
Since the polytope is cut out by the ``cyclic" equalities and inequalities of Proposition \ref{r:inequalities},  (\ref{eq:E_1}) must be a linear combination of cyclic equalities satisfied by $\Gamma_M$; \emph{i.e.}, equalities of the form $\sum_{e \in I} x_e = r_M(I)$ for cyclic intervals $I$. 

If $E_1$ is not a cyclic interval, then we need at least two cyclic equalities different from $\sum_{e \in [n]} x_i = r(M)$ to obtain (\ref{eq:E_1}). Therefore $\Gamma_M$ satisfies at least three linearly independent equations, and $\dim \Gamma_M \leq n-3$. This contradicts the fact \cite{Coxetermatroids} that $\dim \Gamma_M = n-c$ where $c$ is the number of connected components of $M$.
\end{proof}


\begin{definition}
Let $S$ be a partition $[n]=S_1 \sqcup \cdots \sqcup S_t$  of $[n]$ into pairwise disjoint non-empty subsets. We say that $S$ is a \emph{non-crossing partition} if there are no $a,b,c,d$ in cyclic order such that $a,c \in S_i$ and $b,d \in S_j$ for some $i \neq j$. 
Equivalently, place the numbers $1,2,\dots, n$ on $n$ vertices around a circle in clockwise order, and then for each $S_i$, draw a polygon on the corresponding vertices. If no two of these polygons intersect, then $S$ is a non-crossing partition of $[n]$. 

Let $NC_n$ be the set of non-crossing partitions of $[n]$.
\end{definition}

\begin{theorem}\label{th:pos-sum}
Let $M$ be a positroid on $[n]$ and let $S_1$, $S_2$, \dots, $S_t$ be the ground sets of the connected components of $M$. Then $\Pi_M=\{S_1, \ldots, S_t\}$ is a non-crossing partition of $[n]$, called the \emph{non-crossing partition of $M$}.

Conversely, if $S_1$, $S_2$, \dots, $S_t$ form a non-crossing partition of $[n]$ and $M_1$, $M_2$, \dots, $M_t$ are connected positroids on $S_1$, $S_2$, \dots, $S_t$, respectively, then 
$M_1 \oplus \dots \oplus M_t$ is a positroid.
\end{theorem}

\begin{proof}
To prove the first statement of the theorem, 
let us suppose that $S_1$, $S_2$, \dots, $S_t$ do not form a non-crossing 
partition of $[n]$.  Then we can find two parts $S_a$ and $S_b$
and $1 \leq i < j < k < l\leq n$ such that 
$i,k \in S_a$ and $j,l \in S_b$.
But then the restriction of $M$ to $S_a \cup S_b$ is the direct sum of two connected positroids where $S_a$ and $S_b$ are not cyclic intervals.
This contradicts Proposition \ref{prop:pos-sum}.

We prove the second statement of the theorem by induction on $t$,
the number of parts in the non-crossing partition.  
Since $S_1$, \dots, $S_t$ is a non-crossing partition, 
we can assume that one of the parts, say $S_t$, is a cyclic interval
in $[n]$.
Then $S_1$, \dots, $S_{t-1}$ is a non-crossing partition
on $[n]-S_t$. By the inductive hypothesis, 
$M' = M_1 \oplus \dots \oplus M_{t-1}$ is a positroid on $[n]-S_t$.
But now  $M'$ and $M_t$ are positroids on $[n]-S_t$ and $S_t$,
which are cyclic intervals of $[n]$.  Therefore by 
Proposition \ref{directsum}, 
$M = M' \oplus M_t$ is a positroid.
\end{proof}

As remarked earlier, the first half of Theorem \ref{th:pos-sum} was also stated without proof by Oh, Postnikov, and Speyer in \cite{OPS}, and  will also appear in  Ford's preprint \cite{Ford}. The following results, describing direct sums and connectivity in terms of decorated permutations, are also anticipated in \cite{OPS}.
\begin{definition}
Suppose $S_1$ and $S_2$ are disjoint sets. If $\pi_1$ is a decorated permutation of $S_1$ and $\pi_2$ is a decorated permutation of $S_2$, the \emph{direct sum} $\pi_1 \oplus \pi_2$ is the decorated permutation of the set $S_1 \sqcup S_2$ such that $\pi|_{S_1} = \pi_1$ and $\pi|_{S_2} = \pi_2$.
\end{definition}

\begin{proposition}\label{r:permutation_sum}
Let $M_1, \dotsc, M_t$ be positroids on the ground sets $S_1, \dotsc , S_t$, respectively, and suppose $\{ S_1, \dotsc , S_t \}$ is a non-crossing partition of $[n]$. Let $\pi_i$ be the decorated permutation of $S_i$ associated to $M_i$, for $i = 1, \dotsc , t$. Then the decorated permutation associated to the positroid $M_1 \oplus \dotsb \oplus M_t$ is the direct sum $\pi_1 \oplus \dotsb \oplus \pi_t$.
\end{proposition}
\begin{proof}
By induction, it is enough to prove the result when $t=2$. For this purpose, suppose $M_1$ and $M_2$ are positroids on disjoint cyclic intervals $[i,j-~1]$ and $[j,i-1]$ of $[n]$, respectively. Denote by $\pi_1$ and $\pi_2$ their corresponding decorated permutations. Let $(I_i,I_{i+1}, \dotsc , I_{j-1})$ be the Grassmann necklace associated to $M_1$, and let $(J_j,J_{j+1}, \dotsc , J_{i-1})$ be the Grassmann necklace associated to $M_2$. Recall that the bases of $M$ are the disjoint unions of a basis of $M_1$ with a basis of $M_2$. If $l \in [i,j-1]$ then the lexicographically minimal basis of $M$ with respect to the order $<_l$ of $[n]$ is $I_l \sqcup J_j$. Similarly, if $l \in [i,j-1]$, the lexicographically minimal basis of $M$ with respect to $<_l$ is $I_i \sqcup J_l$. It follows that the decorated permutation associated to $M$ is $\pi_1 \oplus \pi_2$.
\end{proof}

\begin{corollary}\label{cor:finest}
 Let $M$ be a positroid on $[n]$, and let $\pi$ be is its corresponding decorated permutation. Then the non-crossing partition $\Pi_M$ associated to $M$ is the finest non-crossing partition of $[n]$ such that for any $i \in [n]$, the numbers $i$ and $\pi(i)$ are in the same block of $\Pi_M$.
\end{corollary}
\begin{proof}
 Let $\{S_1, \ldots, S_t\}$ be the finest non-crossing partition satisfying the described condition. Then for any $i = 1, \dotsc , t$, the decorated permutation $\pi$ restricts to a decorated permutation of $S_i$, which corresponds to a positroid $M_i$ on $S_i$. Since $\pi$ decomposes as $\pi = \pi|_{S_1} \oplus \dotsb \oplus \pi|_{S_t}$, by Proposition \ref{r:permutation_sum} we have $M = M_1 \oplus \dotsb \oplus M_t$. Moreover, each of the positroids $M_i$ is connected, since its decorated permutation $\pi_i$ cannot be decomposed into the direct sum of smaller decorated permutations. The matroids $M_1, \dotsc, M_t$ are then the connected components of $M$, and therefore $\Pi_M =  \{S_1, \ldots, S_t\}$.
\end{proof}

Note that if we represent a decorated permutation of $[n]$ by means of its ``chord diagram'' (see Figure \ref{fig:chord}), Corollary \ref{cor:finest} says that the blocks of its corresponding non-crossing partition are the connected components of the diagram.
\begin{figure}[ht]
  \centering
   \includegraphics[scale=0.9]{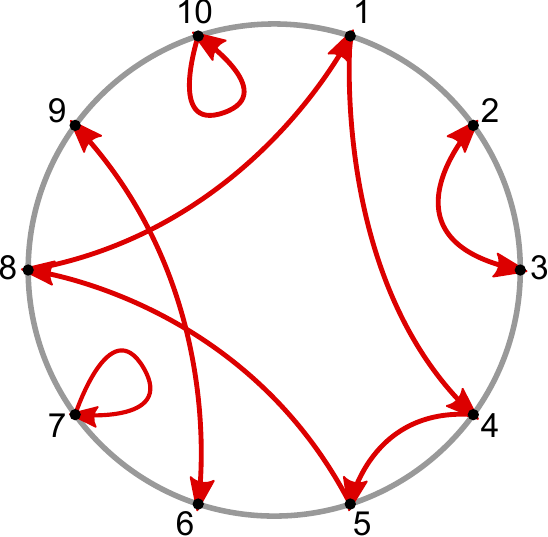}
   \qquad\qquad
   \includegraphics[scale=0.9]{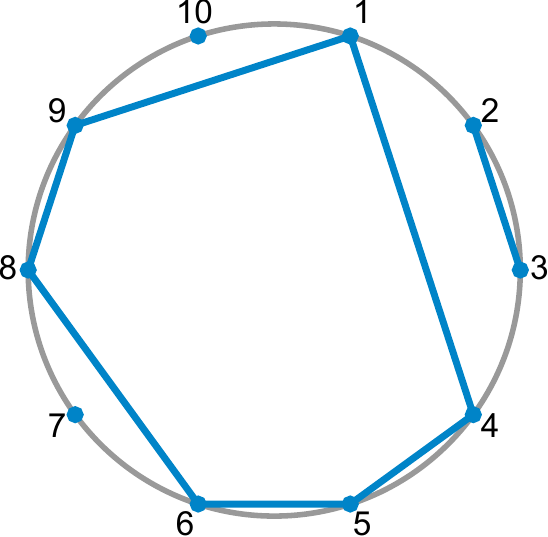}
  \caption{The chord diagram of a decorated permutation of $[10]$ and its corresponding non-crossing partition.}
  \label{fig:chord}
\end{figure}

As a corollary, we obtain a bijection between connected positroids on $[n]$ and an interesting class of permutations of $[n]$.

\begin{definition}\cite{Callan}
A \emph{stabilized-interval-free} (SIF) permutation $\pi$ of $[n]$ is a permutation which does not stabilize any proper interval of $[n]$; that is, $\pi(I) \neq I$ for all intervals $I \subsetneq [n]$.
\end{definition}

\begin{corollary}\label{cor:connSIF}
For $n \geq 2$, the number of connected positroids on $[n]$ equals the number of SIF permutations on $[n]$.
\end{corollary}

\begin{proof}
It follows from Corollary \ref{cor:finest} that a positroid $M$ is connected if and only if its corresponding decorated permutation $\pi$ does not stabilize any proper \emph{cyclic} interval of $[n]$. But $\pi$ stabilizes a cyclic interval $[i,j-1]$ if and only if it stabilizes its complement $[j,i-1]$. Since at least one of these two cyclic intervals is a regular (non-cyclic) interval of $[n]$, we have that $M$ is connected if and only if $\pi$ is SIF.
\end{proof}

\section{A complementary view on positroids and non-crossing partitions}
\label{sec:kreweras}

We now give a complementary description of the non-crossing partition of a positroid, as defined by Theorem \ref{th:pos-sum}. To do that we need the notion of Kreweras complementation.

\begin{definition}
Let $\Pi$ be a non-crossing partition of $[n]$.
Consider nodes $1,1',2,2', \ldots, n,n'$ in that order around a circle, and draw the partition $\Pi$ on the labels $1, 2, \ldots, n$. The \emph{Kreweras complement} $K(\Pi)$ is the coarsest (non-crossing) partition of $[n]$ such that when we regard it as a partition $K(\Pi)'$ of $1', 2', \ldots, n'$, the partition $\Pi \cup K(\Pi)'$ of $1,1',2,2', \ldots, n,n'$ is non-crossing.
\end{definition}
Figure \ref{fig:Kreweras} shows an example of Kreweras complementation.

\begin{figure}[ht]
\begin{center}
 \includegraphics[scale=0.6]{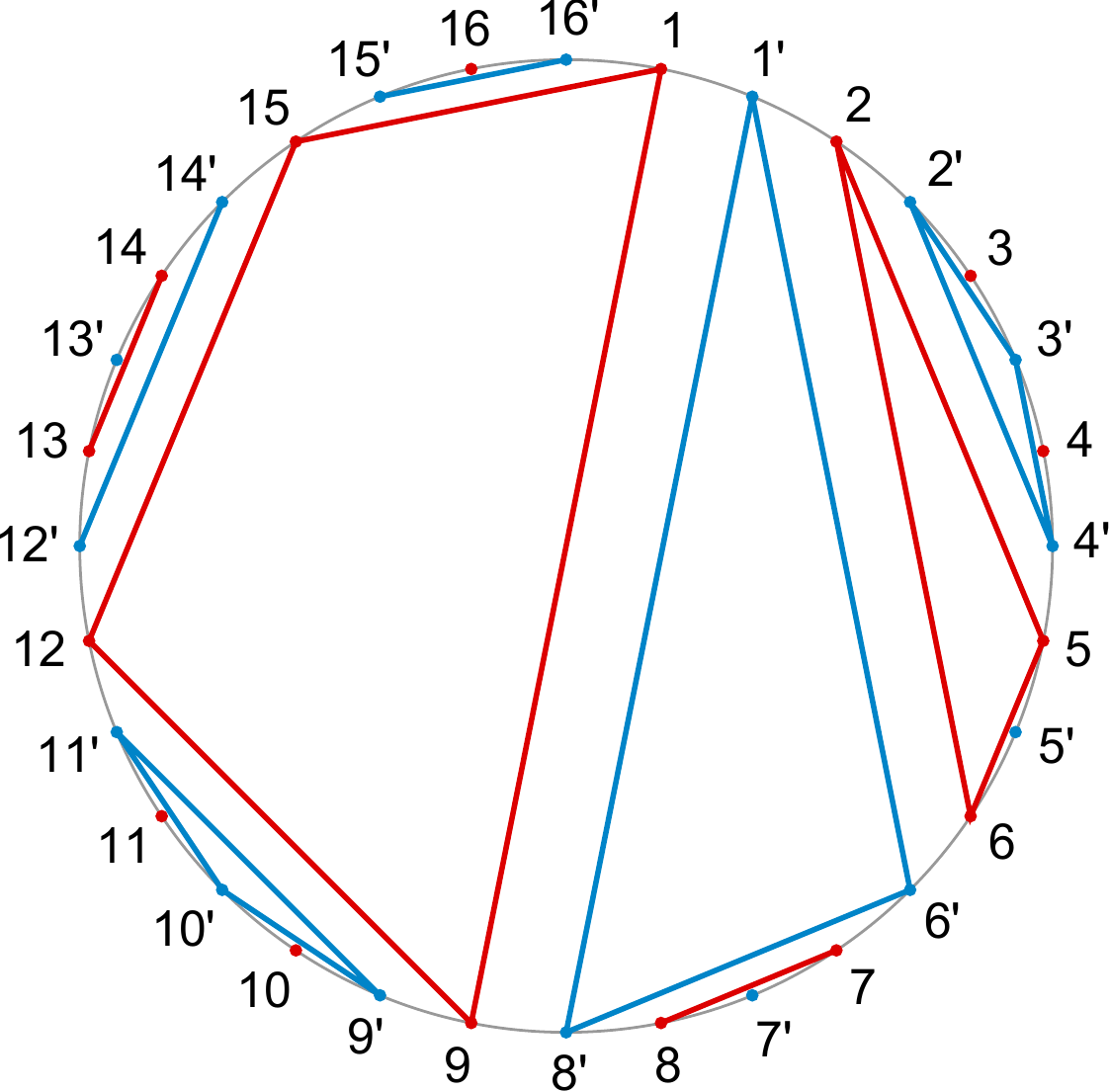}
 \caption{
The Kreweras complement of the (blue) partition {\small $\{\{1,9,12,15\}, \{2,5,6\}, \{3\}, \{4\}, \{7,8\}, \{10\},\{11\}, \{13,14\}, \{16\} \}$} is 
{\small $\{ \{1,6,8\}, \{2,3,4\}, \{5\}, \{7\}, \{9,10,11\}, \{12,14\}, \{13\}, \{15,16\} \}$},
shown in red.}
 \label{fig:Kreweras}
\end{center}
\end{figure}

Let $M$ be a rank $d$ positroid on $[n]$ and consider its matroid polytope $\Gamma_M$ in $\R^n$. Instead of the usual coordinates $x_1, \ldots, x_n$, consider the system of coordinates $y_1, \ldots, y_n$ given by 
\[
y_i = x_1+\cdots + x_i \qquad (1 \leq i \leq n).
\]
Recall that by Proposition \ref{prop:facets}, the inequality description of $\Gamma_M$ has the form
\[
y_n = d, \qquad y_j-y_i \leq r_{ij} \textrm{ for } i \neq j,
\]
where $r_{ij} = r_M([i,j])$. For $i, j \in [n]$, define 
\[
i \sim^* j  \textrm{ if and only if } y_j-y_i \textrm { is constant for } y \in \Gamma_M.
\]
This clearly defines an equivalence relation $\sim^*$ on $[n]$. Let $\Pi^*_M$ be the partition of $[n]$ into equivalence classes of $\sim^*$

\begin{theorem}
The partition $\Pi^*_M$ is the Kreweras complement of the non-crossing partition $\Pi_M$ of $M$. Consequently, it is also non-crossing.
\end{theorem}

\begin{proof}
First we prove that $K(\Pi_M)$ is a refinement of $\Pi_M^*$.
Consider a block $S$ of $K(\Pi_M)$ and two cyclically consecutive elements $i<j$ in $S$. Since $\Pi_M \cup K(\Pi_M)'$ is non-crossing in $1,1', \ldots, n,n'$, the cyclic interval $[i+1,j]$ of $[n]$ is a disjoint union of blocks $S_1, \ldots, S_s$ of $\Pi_M$, which are themselves connected components of $M$. If $1 \leq i<j$ in cyclic order, we have that
\[
y_j-y_i = \sum_{a \in [i+1, j]} x_a = \sum_{r=1}^s \sum_{a \in S_r} x_a  = \sum_{r=1}^s r(S_r)
\]
is constant in $\Gamma_M$, and therefore $i \sim^* j$. A similar computation holds if $i< 1 \leq j$. It follows that $K(\Pi_M)$ is a refinement of $\Pi_M^*$.

Now assume that $i \sim^*j$ but $i$ and $j$ are not in the same block of $K(\Pi_M)$. 
Looking at the non-crossing partition $\Pi_M \cup K(\Pi_M)'$, this means that the edge $i'j'$ (which is not in $K(\Pi_M)$) must cross an edge $kl$ of $\Pi_M$. Assume $k \in [i+1,j]$ and $l \notin [i+1,j]$. Now, since $k \sim l$, we can find bases $B$ and $B'$ of $M$ with $B'=B - \{k\} \cup \{l\}$. 
But then 
$\sum_{a \in [i+1, j]}x_a$ is not constant on $\Gamma_M$: more specifically,
the value it takes on the vertex $e_B$ is $1$ more than the value
it takes on the vertex $e_{B'}$.
This contradicts the fact that $i \sim^*j$.
%
\end{proof}

\section{Positroid polytopes and non-crossing partitions}\label{sec:embed}

Having explained the role that non-crossing partitions play in the connectivity of positroids, we use that knowledge to show that the face poset of a positroid polytope lives inside the poset of \emph{weighted non-crossing partitions.}

\begin{definition}
A \emph{weighted non-crossing partition} $S^w$ of $[n]$ is a non-crossing partition $S$ of $[n]$, say $[n]=S_1 \sqcup \cdots \sqcup S_t$, together with a weight vector $w=(w_1,\dots,w_t) \in (\ZZ_{\geq0})^t$ of integer 
weights $w_1 = w(S_1), \dots, w_t = w(S_t)$ with $0 \leq w_i \leq |S_i|$ for $i=1, \ldots, t$. The \emph{weight} of the partition $S^w$ is $w_1+\cdots+w_t$. 
\end{definition}

The set $NC_n$ of non-crossing partitions of $[n]$ is partially ordered by refinement; this poset has many interesting properties and connections to several fields of mathematics. We extend that order to the context of weighted non-crossing partitions. 

\begin{definition}
Let $NC_n^d$ be the poset of non-crossing partitions of $[n]$ of weight $d$, where the cover relation is given by $S^w \lessdot T^v$ if 

\noindent $\bullet$ $T=\{T_1, \ldots, T_t\}$ and $S=\{T_1, \ldots, T_{h-1}, A, T_h - A, T_{h+1}, \ldots, T_t\}$ for some index $1 \leq h \leq t$ and some proper subset $\emptyset \subsetneq A \subsetneq T_h$, and

\noindent $\bullet$ $v(T_h) = w(A)+w(T_h - A)$ and $v(T_j) = w(T_j)$ for all $j \neq h$.
%
\noindent Let $NC_n^d \cup \{\widehat{0}\}$ be this poset with an additional minimum element $\widehat{0}$.
\end{definition}

The poset $NC_n^d$ is ranked of height $n$. It has a unique maximal element $\widehat{1}$ corresponding to the trivial partition of $[n]$ into one part of weight $d$. 

Readers familiar with 
the \emph{poset $\Pi_n^w$ of weighted partitions} defined by Dotsenko and Khorsohkin \cite{DK} and further studied by Gonz\'alez and Wachs \cite{GW} may notice the relationship between these two posets. 
The subposet of $\Pi_n^w$ consisting of the non-crossing partitions of weight $d$ is almost equal to $NC_n^d$; the only difference is that $0 \leq w(S_i) \leq |S_i|-1$ in $\Pi_n^w$ and $0 \leq w(S_i) \leq |S_i|$ in $NC_n^d$. For our purposes we only need to allow $w(S_i) = |S_i|$ for $|S_i|=1$, but this small distinction is important; see Remark \ref{rem:weighted}.

\begin{theorem}\label{thm:embed}
If $M$ is a rank $d$ positroid on $[n]$ then the face poset of the matroid polytope $\Gamma_M$ is an induced subposet of $NC_n^d \cup \{\widehat{0}\}$.
\end{theorem}

\begin{proof}
By Corollary \ref{cor:faceofpositroid}, any non-empty face $F$ of the positroid polytope $\Gamma_M$ is itself a positroid polytope, say $F=\Gamma_{N}$. Write $N = N_1 \oplus \cdots \oplus N_c$ as a direct sum of its connected components. By Theorem \ref{th:pos-sum}, the partition $\Pi_F=\{N_1, \ldots, N_c\}$ of $[n]$ is non-crossing. 
Assign weights $w(N_i) = r_N(N_i)$ for $1 \leq i \leq c$ to the blocks of this partition. Let $\Pi_F^w$ be the resulting weighted non-crossing partition. 
Since $r_N(N_1) + \cdots + r_N(N_c) = r(N) = r(M) = d$ we have that $\Pi_F^w \in NC_n^d$. (If $F$ is the empty face let $\Pi_F^w = \widehat{0}$.) We claim that 
\[
F \mapsto \Pi_F^w
\] 
is the desired embedding.

Figure \ref{fig:NCposet} shows the positroid polytope $\Gamma_M$ for the positroid $M$ whose bases are $\{12, 13, 14, 23, 24\}$. It is a square pyramid. It also shows the face poset of $\Gamma_M$, with each face labeled with the corresponding weighted non-crossing partition of $[4]$.

\begin{figure}[ht]
\begin{center}
 \includegraphics[width=12cm]{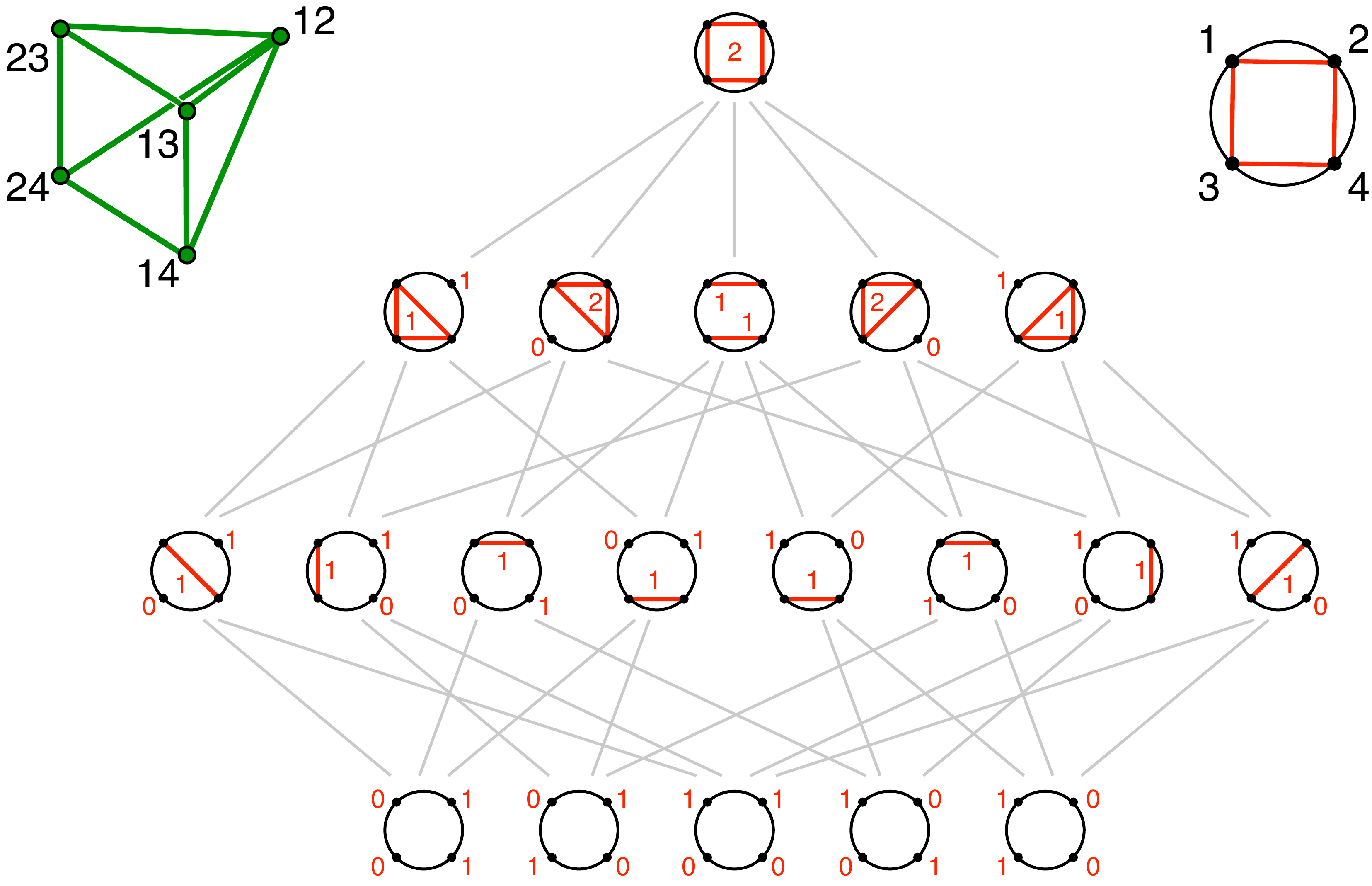}
 \caption{The face poset of the square pyramid inside $NC_4^2$.}
 \label{fig:NCposet}
\end{center}
\end{figure}

First we show that this mapping is one-to-one. Suppose we know $\Pi_F^w$ 
and we wish to recover $F$. Since $F$ is a face of $\Gamma_M$, it satisfies the same inequalities as $\Gamma_M$, 
and some additional equalities. If $F=\Gamma_N$, the equalities that it satisfies are
$\sum_{i \in N_j} x_i = r_N(N_j)$ for $j = 1, \ldots, c$ and their linear combinations. But we know the $N_j$s and the $r_N(N_j)$s from $\Pi_F^w$, so we can recover $F$ as the intersection of $\Gamma_M$ with these $c$ hyperplanes.

Now we show that the mapping is order-preserving. Assume that $F \gtrdot G$ are faces of $\Gamma_M$; say $F=\Gamma_K$ and $G=\Gamma_L$ for positroids $K$ and $L$. Let $K = K_1 \oplus \cdots \oplus K_c$ be the decomposition of $K$ into connected components. Then $\dim F = n-c$ implies $\dim G = \dim F - 1 = n-c-1$. By Proposition \ref{prop:face}, the decomposition of $L$ into connected components must then be of the form $L = K_1 \oplus \cdots \oplus K_{h-1} \oplus (K_h|A) \oplus (K_h/A) \oplus K_{h+1} \oplus \cdots \oplus K_c$ for some $1 \leq h \leq c$ and some proper subset $A \subset K_h$. Therefore $\Pi_G \lessdot \Pi_F$ in $NC_n$. Furthermore, since $K_j$ has the same weight in $\Pi_F$ and $\Pi_G$ for all $j \neq i$ and $r(K)=r(L)=d$, the weight $r_K(K_h)$ in $\Pi_F$ must equal the sum $r_L(K_h|A) + r_L(K_h/ A)$ of weights in $\Pi_G$. Therefore  $\Pi_G^w \lessdot \Pi_F^w$ in $NC_n^d$. 

Finally, to show that the  face poset of $\Gamma_M$ is embedded as an \emph{induced} subposet of $NC_n^d$, assume that $\Pi_G^w \leq \Pi_F^w$ for some faces $F$ and $G$ of $\Gamma_M$. We need to show that $G \leq F$. Again let $F=\Gamma_K$ and $G=\Gamma_L$, and let $K_1, \ldots, K_c$ be the components of $K$. 
The components of $L$ must be a refinement of the components of $K$, say $K^1_1, K^2_1, \dotsc, K^{i_1}_1$, $K^1_2, K^2_2, \dotsc, K^{i_2}_2$, $\ldots$, $K^1_c, K^2_c, \dotsc, K^{i_c}_c$, where $\bigsqcup_{j = 1}^{i_l} K^j_l = K_l$, for $l = 1, \ldots, c$. Moreover, we have $\sum_{j = 1}^{i_l} r_L(K^j_l) = r_K(K_l)$. Now, the equalities that determine the face $F$ as a subset of $\Gamma_M$ are $\sum_{j \in K_l} x_j = r_K(K_l)$ for $l = 1, \ldots, c$. 
The face $G$ is cut out of $\Gamma_M$ by the equalities $\sum_{j \in K^i_l} x_j = r_L(K^i_l)$ for $l = 1, \ldots, c$ and $i = 1, \ldots, i_l$. These latter inequalities easily imply the inequalities that describe $F$ in $\Gamma_M$, so it follows that $G \leq F$.
\end{proof}

\begin{remark}\label{rem:weighted}
In the correspondence above, the weight of a block $N_i$ in a non-crossing partition $\Pi_F^w$ is $w(N_i) = r_N(N_i)$. If we had $r_N(N_i) = |N_i|$, then $N_i$ would consist solely of coloops. Since $N_i$ is connected, we must have $|N_i| \in \{0,1\}$. Singletons may indeed have weight equal to $0$ or $1$. This is the only reason why, in $NC_n^d$, we need to allow a block of size $k$ to have weight $k$, instead of following \cite{DK, GW}. 
\end{remark}

As mentioned earlier, the poset $NC_n^d \cup \{\widehat{0}\}$ is ranked of height $n$. The face poset $F(\Gamma_M)$ of the polytope $\Gamma_M$ of a connected positroid $M$ of rank $d$ on $[n]$ is also ranked of height $n$. 
For each such positroid $M$, the order complex $\Delta(F(\Gamma_M) - \{\widehat{0}, \widehat{1}\})$ can be identified with the barycentric subdivision of the polytope $\Gamma_M$, so it is homeomorphic to an $(n-2)$-sphere. The interaction of these different $(n-2)$-spheres inside the order complex $\Delta(NC_n^d - \{\widehat{1}\})$ is the subject of an upcoming project.

\section{Enumeration of connected positroids}\label{sec:enumeration}

In this section we use Theorem \ref{th:pos-sum}, together with a result of the third author \cite{Williams}, 
to enumerate connected positroids.

\begin{definition}
Let $p(n)$ be the number of positroids on $[n]$ and $p_c(n)$ be the number of connected positroids on $[n]$. Let
\[
P(x) = 1+\sum_{n \geq 1} p(n)x^n \quad \text{ and }\quad  P_c(x) = 1+\sum_{n \geq 1} p_c(n)x^n.
\]
\end{definition}

Many combinatorial objects (such as graphs or matroids) on a set $[n]$ decompose uniquely into connected components $S_1, \ldots, S_k$, where the partition $[n]=S_1 \sqcup \cdots \sqcup S_k$ has no additional structure. In that case, the Exponential Formula
\cite[Theorem 5.1.3]{EC2}
tells us that the exponential generating functions $E_t(x)$ and $E_c(x)$ for the total number of objects and the total number of connected objects are related by the formula $E_c(x) = \log E_t(x)$.

In our situation, where the connected components of a positroid
form a non-crossing partition, we need the following ``non-crossing" analog of 
the Exponential Formula:

\begin{theorem}\label{th:speicher}\cite{Speicher}
Let $K$ be a field. Given a function $f:\ZZ_{>0} \rightarrow K$ define a new function $h:\ZZ_{>0} \rightarrow K$ by
\begin{equation}\label{eq:NC}
h(n) = \sum_{\{S_1, \ldots, S_k\} \in NC_n} f(\#S_1)f(\#S_2) \cdots f(\#S_k),
\end{equation}
where we are summing over all the non-crossing partitions of $[n]$. Define  $F(x) = 1+\sum_{n \geq 1} f(n)x^n$ and $H(x) = 1+\sum_{n \geq 1} h(n)x^n$. Then
\[
xH(x) = \left(\frac{x}{F(x)}\right)^{\langle -1 \rangle},
\]
where $G(x)^{\langle -1 \rangle}$ denotes the compositional inverse of $G(x)$.
\end{theorem}

\begin{corollary} \label{cor:pos} The generating functions for positroids and connected positroids satisfy:
\[
xP(x) = \left( \frac{x}{P_c(x)} \right)^{\langle -1 \rangle}.
\]
\end{corollary}

\begin{proof} 
Theorem \ref{th:pos-sum} implies that
\[
p(n) = \sum_{\{S_1, \ldots, S_k\} \in NC_n} p_c(\#S_1) p_c(\#S_2) \cdots p_c(\#S_k),
\]
and Theorem \ref{th:speicher} then gives the desired result.
\end{proof}

Enumeration of general positroids has been previously studied by the third author in \cite{Williams}.

\begin{theorem}\label{th:count}
We have
\[
P(x) =  \sum_{k \geq 0} k! \frac{x^k}{(1-x)^{k+1}}, \qquad
p(n) = \sum_{k=0}^n \frac{n!}{k!}, \qquad 
\lim_{n \rightarrow \infty} \frac{p(n)}{n!} = e.
\]
\end{theorem}

\begin{proof}
In \cite{Williams}, Williams gave a finer enumeration of positroids in terms of the size of the ground set, the rank, and the dimension of the positroid cell. 
The first equality follows from \cite[Prop. 5.11]{Williams} by setting $q=y=1$. This easily implies the second equality, which implies the third.
\end{proof}

The following formula also follows easily from the above.
\begin{proposition}\cite[Prop. 23.2]{postnikov}\label{prop:expcount}
The exponential generating function for $p(n)$ is
\[
1+\sum_{n \geq 1} p(n) \frac{x^n}{n!} = \frac{e^x}{1-x}.
\]
\end{proposition}

The sequence $\{p(n)\}_{n \geq 1}$ is entry A000522 in Sloane's Encyclopedia of Integer Sequences \cite{Sloane}. The first few terms are $2, 5, 16, 65, 326, 1957$, $13700, \ldots$ .

\begin{theorem}\label{th:countconn}
The number $p_c(n)$ of connected positroids on $[n]$ satisfies
\begin{eqnarray*}
p_c(n) &=& \frac{[x^n] P(x)^{1-n}}{1-n}, \\
p_c(n) &=& (n-1)p_c(n-1) + \sum_{j=2}^{n-2} (j-1)p_c(j)p_c(n-j) \,\, \text{for }n \geq 2, \text{ and} \\
\lim_{n \rightarrow \infty} \frac{p_c(n)}{n!} &=& \frac1e.
\end{eqnarray*}
\end{theorem}

\begin{proof}
The first statement follows by applying the Lagrange inversion formula \cite[Theorem 5.4.2]{EC2} to $F(x)=x/P_c(x)$ and $F^{\langle -1 \rangle}(x) = xP(x)$, which says:
\begin{eqnarray*}
m[x^m]\left(F^{\langle -1 \rangle}(x)\right)^k &=& k[x^{-k}] F(x)^{-m} \\
m[x^m]\left(xP(x)\right)^k &=& k[x^{-k}] \left(\frac{P_c(x)}x\right)^m \\
m[x^{m-k}] P(x)^k &=& k[x^{m-k}] P_c(x)^m.
\end{eqnarray*}
It remains to set $m=1$ and $k=1-n$. 

In view of Corollary \ref{cor:connSIF}, the second statement is derived in \cite{Callan}, and the third is a consequence of \cite[Cor. 11]{ST}. 
\end{proof}

The sequence $\{p_c(n)\}_{n \geq 1}$ is, except for the first term, equal to entry A075834 in Sloane's Encyclopedia of Integer Sequences \cite{Sloane}. The first few terms are $2, 1, 2, 7, 34, 206, 1476, \ldots$ .

We conclude the following.

\begin{theorem}\label{th:ratio}
If $p(n)$ is the number of positroids on $[n]$ and $p_c(n)$ is the number of connected positroids on $[n]$, then
\[
\lim_{n \rightarrow \infty} \frac{p_c(n)}{p(n)} = \frac1{e^2} \approx 0.1353.
\]
\end{theorem}

\begin{proof}
This is an immediate consequence of Theorems \ref{th:count} and \ref{th:countconn}.
\end{proof}

This result is somewhat surprising in view of the conjecture that most matroids are connected:

\begin{conjecture}\label{conj:conn}(Mayhew, Newman, Welsh, Whittle, \cite{MNWW})
If $m(n)$ is the number of matroids on $[n]$ and $m_c(n)$ is the number of connected matroids on $[n]$, then
\[
\lim_{n \rightarrow \infty} \frac{m_c(n)}{m(n)} = 1.
\]
\end{conjecture}

Theorem \ref{th:ratio} should \textbf{not} be seen as evidence against Conjecture \ref{conj:conn}. 
Positroids possess strong structural properties that are quite specific to them.
Furthermore, they  are a relatively small family of matroids: compare the 
estimate 
$\log_2 \log_2 m(n) \sim n$ 
due to Knuth \cite{Knuth} and Bansal, Pendavingh, and van der Pol \cite{BPV} with the estimate $p(n) \sim n! \, e $, which gives $\log_2 \log_2 p(n) \sim \log_2 n$.

\section{Positroids and free probability}\label{sec:free}

The results of the previous section have an interesting connection with Voiculescu's theory of \emph{free probability}. We give a very brief overview of the aspects of the theory that are relevant to our discussion; for a more thorough introduction, we recommend Speicher's excellent survey \cite{Speicher}.

The concept of \emph{freeness} can be thought of as a ``non-commutative analogue" to the classical notion of independence in probability.  The role played by independence, moments, cumulants, and partitions in classical probability is now played by freeness, moments, free cumulants, and non-crossing partitions in free probability, as we now explain.

Given a real-valued random variable $X$ with probability distribution $\mu(x)$, the \emph{moments} of $X$ are the expected values of the powers of $X$: the $n$th moment is $m_n(X) = E(X^n)$,
for $n \geq 1$. (We assume for the rest of this discussion that all moments exist.) The \emph{moment generating function} 
\[
M_X(t) = E(e^{tX}) = \sum_{n \geq 0} m_n(X) \frac{t^n}{n!}
\]
is essentially the same as the Fourier transform of $\mu$. The \emph{cumulants} of $X$ are the coefficients of the generating function
\[
\log M_X(t) =  \sum_{n \geq 0} c_n(X) \frac{t^n}{n!}.
\]
The independence of random variables $X$ and $Y$ translates into a linear relation of cumulants. Since expectation is multiplicative on independent variables, we have that $M_{X+Y}(t) = M_X(t)M_Y(t)$ when $X$ and $Y$ are independent, so
\[
X, Y \textrm{ independent} \Rightarrow c_n(X+Y) = c_n(X) + c_n(Y) \textrm{ for all } n \geq 1.
\]

In the non-commutative setting, our ``random variables" are simply elements of a unital algebra $\A$ which is not necessarily commutative. Our ``expectation" $E$ is just a linear function $E: \A \rightarrow \C$ with $E(1)=1$. Moments are defined in analogy with the classical case. We say that random variables $X$ and $Y$ are \emph{free} if, 
for any polynomials $p_1, q_1, \ldots, p_k, q_k$,
\[
E(p_i(X)) = E(q_j(Y)) = 0 \textrm{ for all } i,j \, \Rightarrow \, 
E(p_1(X)q_1(Y) \cdots p_k(X)q_k(Y)) = 0.
\]

Again, the freeness of $X$ and $Y$ manifests linearly in terms of the 
\emph{free cumulants}, which are the numbers $k_1, k_2, \ldots$ such that
\begin{equation}
m_n = \sum_{\{S_1, \ldots, S_k\} \in NC_n} k_{\#S_1}k_{\#S_2} \cdots k_{\#S_k} \label{eq:m,k}
\end{equation}
for all $n$. 
While the formula for the moments of $X+Y$ is quite intricate, the free cumulants are related beautifully by: 
\[
X, Y \textrm{ free } \Rightarrow k_n(X+Y) = k_n(X) + k_n(Y) \textrm{ for all } n \geq 1.
\]
There is also a remarkable formula for the free cumulants of $X \cdot Y$ \cite{Speicher}.

With all the necessary background in place, we can now establish a simple connection between free probability and  positroids. Let $\textrm{Exp}(\lambda)$ be an exponential random variable with rate parameter $\lambda$.

\begin{theorem}
The moments of the random variable $Y \sim 1+ \textrm{Exp}(1)$ are
\[
m_n(Y) = \# \textrm{ positroids on } [n],
\]
and its free cumulants are
\[
k_n(Y) = \# \textrm{ connected positroids on } [n].
\]
\end{theorem}

\begin{proof}
Using the fact that 
$M_{A+B}(t)=M_A(t)M_B(t)$ for independent random variables $A$ and $B$, and that $M_{\textrm{Exp}(\lambda)} = 1/(1-\frac{t}{\lambda})$, it follows that 
the moment generating function of $Y$ is
\[
M_Y(t) = M_{1+\textrm{Exp}(1)}(t) = M_1(t)M_{\textrm{Exp}(1)}(t) = e^t \cdot \frac{1}{1-t}.
\]

Comparing with Proposition \ref{prop:expcount} gives the first formula. The second follows by combining Corollary \ref{cor:pos} with the relation 
\eqref{eq:m,k} between moments and free cumulants.
\end{proof}

\bibliographystyle{amsalpha}
\bibliography{bibliography}

\def\cprime{$'$}
\providecommand{\bysame}{\leavevmode\hbox to3em{\hrulefill}\thinspace}
\providecommand{\MR}{\relax\ifhmode\unskip\space\fi MR }
\providecommand{\MRhref}[2]{%
  \href{http://www.ams.org/mathscinet-getitem?mr=#1}{#2}
}
\providecommand{\href}[2]{#2}
\begin{thebibliography}{MNWW11}

\bibitem[AK06]{AK}
Federico Ardila and Caroline~J. Klivans, \emph{The {B}ergman complex of a
  matroid and phylogenetic trees}, J. Combin. Theory Ser. B \textbf{96} (2006),
  no.~1, 38--49.

\bibitem[BGW03]{Coxetermatroids}
A.~Borovik, I.~Gelfand, and N.~White, \emph{Coxeter matroids}, Birkh\"auser,
  2003.

\bibitem[BPvdP]{BPV}
N.~Bansal, R.A. Pendavingh, and J.G. van~der Pol, \emph{On the number of
  matroids}, Preprint. {\tt arXiv:1206.6270}.

\bibitem[Cal04]{Callan}
David Callan, \emph{Counting stabilized-interval-free permutations}, J. Integer
  Seq. \textbf{7} (2004), no.~1, Article 04.1.8, 7 pp. (electronic).

\bibitem[DK07]{DK}
V.~V. Dotsenko and A.~S. Khoroshkin, \emph{Character formulas for the operad of
  a pair of compatible brackets and for the bi-{H}amiltonian operad},
  Funktsional. Anal. i Prilozhen. \textbf{41} (2007), no.~1, 1--22, 96.

\bibitem[For13]{Ford}
Nicolas Ford, \emph{The expected codimension of a matroid variety}, Preprint,
  2013.

\bibitem[GDW13]{GW}
Rafael Gonzalez~D'Leon and Michelle Wachs, \emph{On the poset of weighted
  partitions}, 25th {A}nnual {I}nternational {C}onference on {F}ormal {P}ower
  {S}eries and {A}lgebraic {C}ombinatorics ({FPSAC} 2013), Discrete Math.
  Theor. Comput. Sci. Proc., AJ, Assoc. Discrete Math. Theor. Comput. Sci.,
  Nancy, 2013, pp.~1059--1070.

\bibitem[GGMS87]{GGMS}
Israel~M. Gelfand, R.~Mark Goresky, Robert~D. MacPherson, and Vera~V.
  Serganova, \emph{Combinatorial geometries, convex polyhedra, and {S}chubert
  cells}, Adv. in Math. \textbf{63} (1987), no.~3, 301--316.

\bibitem[Knu74]{Knuth}
Donald~E Knuth, \emph{The asymptotic number of geometries}, Journal of
  Combinatorial Theory, Series A \textbf{16} (1974), no.~3, 398--400.

\bibitem[LP]{LP}
T.~Lam and A.~Postnikov, \emph{Polypositroids}, In progress.

\bibitem[MNWW11]{MNWW}
Dillon Mayhew, Mike Newman, Dominic Welsh, and Geoff Whittle, \emph{On the
  asymptotic proportion of connected matroids}, European J. Combin. \textbf{32}
  (2011), no.~6, 882--890.

\bibitem[Oh11]{oh}
Suho Oh, \emph{Positroids and {S}chubert matroids}, Journal of Combinatorial
  Theory, Series A \textbf{118} (2011), no.~8, 2426--2435.

\bibitem[OPS]{OPS}
S.~Oh, A.~Postnikov, and D.~Speyer, \emph{Weak separation and plabic graphs},
  Preprint. {\tt arXiv:1109.4434}.

\bibitem[Oxl92]{Oxley}
James~G. Oxley, \emph{Matroid theory}, Oxford University Press, 1992.

\bibitem[Pos]{postnikov}
Alexander Postnikov, \emph{Total positivity, {G}rassmannians, and networks},
  Preprint. Available at
  \url{http://www-math.mit.edu/~apost/papers/tpgrass.pdf}.

\bibitem[Pos12]{postnikovpersonal}
Alex Postnikov, \emph{personal communication}, 2012.

\bibitem[PSW09]{PSW}
Alexander Postnikov, David Speyer, and Lauren Williams, \emph{Matching
  polytopes, toric geometry, and the totally non-negative {G}rassmannian}, J.
  Algebraic Combin. \textbf{30} (2009), no.~2, 173--191.

\bibitem[Sch86]{Schrijver}
Alexander Schrijver, \emph{Theory of linear and integer programming}, John
  Wiley \& Sons, Inc., New York, NY, USA, 1986.

\bibitem[Slo94]{Sloane}
N.~J.~A. Sloane, \emph{An on-line version of the encyclopedia of integer
  sequences}, Electron. J. Combin \textbf{1} (1994), 1--5.

\bibitem[Spe94]{Speicher}
Roland Speicher, \emph{Multiplicative functions on the lattice of non-crossing
  partitions and free convolution}, Mathematische Annalen \textbf{298} (1994),
  611--628 (English).

\bibitem[ST09]{ST}
Paolo Salvatore and Roberto Tauraso, \emph{The operad {L}ie is free}, Journal
  of Pure and Applied Algebra \textbf{213} (2009), no.~2, 224--230.

\bibitem[Sta99]{EC2}
Richard~P. Stanley, \emph{Enumerative combinatorics. {V}ol. 2}, Cambridge
  Studies in Advanced Mathematics, vol.~62, Cambridge University Press,
  Cambridge, 1999.

\bibitem[Tal08]{Talaska}
Kelli Talaska, \emph{A formula for {P}l\"ucker coordinates associated with a
  planar network}, Int. Math. Res. Not. IMRN \textbf{2008} (2008).

\bibitem[Wel76]{welsh}
Dominic J.~A. Welsh, \emph{Matroid theory}, Academic Press [Harcourt Brace
  Jovanovich Publishers], London, 1976, L. M. S. Monographs, No. 8.

\bibitem[Wil05]{Williams}
Lauren~K. Williams, \emph{Enumeration of totally positive {G}rassmann cells},
  Advances in Mathematics \textbf{190} (2005), no.~2, 319--342.

\end{thebibliography}
\label{sec:biblio}

\end{document}